\documentclass[11pt]{amsart}
\usepackage{amsmath,amssymb,amsthm,mathrsfs,enumerate,bm,xcolor,graphicx,multirow,pbox}
\usepackage[colorlinks,linkcolor=red,citecolor=blue,urlcolor=blue]{hyperref}
\usepackage{natbib}
\usepackage{placeins}

\numberwithin{equation}{section}

\newcommand{\RR}{\mathbb{R}}
\newcommand{\VV}{\mathbb{V}}
\newcommand{\QQ}{\mathbb{Q}}

\renewcommand{\epsilon}{\varepsilon}

\def\st{\ensuremath{s^*}}

\newcommand{\beq}{\begin{equation}}
\newcommand{\eeq}{\end{equation}}
\newcommand{\beqa}{\begin{equation} \begin{aligned}}
\newcommand{\eeqa}{\end{aligned} \end{equation}}
\newcommand{\beqas}{\begin{equation*} \begin{aligned}}
\newcommand{\eeqas}{\end{aligned} \end{equation*}}

\newcommand{\bit}{\begin{itemize}}
	\newcommand{\eit}{\end{itemize}}
\newcommand{\bmat}{\begin{bmatrix}}
	\newcommand{\emat}{\end{bmatrix}}

\theoremstyle{definition}\newtheorem{definition}{Definition}
\theoremstyle{remark}

\theoremstyle{remark}\newtheorem{remark}{Remark}
\theoremstyle{definition}\newtheorem{example}{Example}
\theoremstyle{plain}\newtheorem{question}{Question}
\theoremstyle{plain}\newtheorem{theorem}{Theorem}
\theoremstyle{plain}
\theoremstyle{plain}\newtheorem{proposition}{Proposition}
\theoremstyle{plain}\newtheorem{corollary}{Corollary}
\theoremstyle{plain}

\begin{document}
	
%

%

\title{Bi-$s^*$-concave distributions} 

\author[N. Laha]{Nilanjana Laha}
\address[N. Laha]{
Department of Statistics, Box 354322, University of Washington, Seattle, WA 98195-4322, USA.
}
\email{nlaha@uw.edu}

\author[J. A. Wellner]{Jon A. Wellner}
\address[J. A. Wellner]{
Department of Statistics, Box 354322, University of Washington, Seattle, WA 98195-4322, USA.
}
\email{jaw@stat.washington.edu}

\thanks{Supported in part by NSF Grant DMS-1566514. }

\date{\today}

\maketitle

\begin{abstract}
We introduce a new shape-constrained class of distribution 
functions on $\RR$, the {\sl bi-$s^*$-concave} class.   In parallel to 
results of \cite{DuembgenKW:2017}  
for what they called the class of bi-log-concave distribution functions,
we show that every $s-$concave density $f$ has a bi-$s^*$-concave 
distribution function $F$ and that every bi-$s^*$-concave distribution function 
satisfies $\gamma (F) \le  1/(1+s)$  where finiteness of
$$
\gamma (F) \equiv \sup_{x} F(x) (1-F(x)) \frac{| f' (x)|}{f^2 (x)},
$$
 the Cs\"org\H{o} - R\'ev\'esz constant of $F$, plays an important role in the 
 theory of quantile processes on $\RR$. 
\end{abstract}

\tableofcontents

\section{Introduction:  the bi-log-concave class}
\label{sec:intro}

\cite{DuembgenKW:2017} 
investigated a shape constraint they called 
``bi-log-concavity'' for distribution functions $F$ on $\RR$:  a distribution function $F$ is 
{\sl bi-log-concave} if both $x \mapsto \log F(x) $ and $x \mapsto \log (1-F(x))$ are concave 
functions of $x$.  They noted 
that 
\cite{MR2213177} 
showed that any log-concave distribution with density $f$ 
has a bi-log-concave distribution function $F$, but that the inclusion is proper:  there are 
many bi-log-concave distributions that are not log-concave, and in fact bi-log-concave distributions
may not be unimodal.   
\cite{DuembgenKW:2017} 
proved the following interesting theorem characterizing 
the class of bi-log-concave distributions.
\medskip

First a bit of notation:  
$$
J(F) \equiv \{ x \in \RR : \ 0 < F(x) < 1 \} .
$$
A distribution function $F$ is non-degenerate if $J(F) \ne \emptyset $.
\smallskip
 
\begin{theorem}
\label{thm:DKWthm} (DKW, 2017)
For a non-degenerate distribution function $F$ the following 
four statements are equivalent:\\
(i) \ $F$ is bi-log-concave.\\
(ii) \ $F$ is continuous on $\RR$ and differentiable on $J(F)$ with derivative $f=F^{\prime}$ such that
\begin{eqnarray*}
F(x+t) \left \{ \begin{array}{l}  \le F(x) \exp \left ( \frac{f(x)}{F(x)} t \right ) \\ \ge 1 - (1-F(x)) \exp \left ( - \, \frac{f(x)}{(1-F(x))}  t \right ) 
\end{array} \right .
\end{eqnarray*}
for all $x \in J(F)$ and $t \in \RR$.\\
(iii) $F$ is continuous on $\RR$ and differentiable on $J(F)$ with derivative $f = F^{\prime}$ such that the 
hazard function $f/(1-F)$ is non-decreasing and reverse hazard function $f/F$ is non-increasing on $J(F)$.\\
(iv) $F$ is continuous on $\RR$ and differentiable on $J(F)$ with bounded 
and strictly positive derivative $f = F^{\prime}$.  Furthermore, $f$ is locally Lipschitz-continuous on $J(F)$ with 
$L^1-$derivative  $f^{\prime} = F^{\prime \prime} $ satisfying 
\begin{eqnarray*}
\frac{-f^2}{1-F} \le f^{\prime} \le \frac{f^2}{F} .
\end{eqnarray*}
\end{theorem}

An important implication of (iv) of Theorem~\ref{thm:DKWthm} is that the inequalities can be rewritten as follows:
\begin{eqnarray*}
-1 \le -F(x)  \le F(x) (1-F(x)) \frac{f^{\prime} (x)}{f^2 (x)} \le 1-F(x)  \le 1 .
\end{eqnarray*}
This implies that the bi-log-concave family of distributions satisfies  
\begin{eqnarray}
\gamma (F) \equiv \sup_{x \in J(F)}  F(x) (1-F(x)) \frac{ | f^{\prime} (x) |}{f^2 (x)} \le 1.
\label{GammaFboundedbyOneForLogConcaveF}
\end{eqnarray}
The parameter $\gamma (F)$ arises in the study of quantile processes and transportation  distances between
empirical distributions and true distributions on  $\RR$:  see e.g. 
\cite{MR0501290},  
\cite{MR838963, MR3396731} Chapter 18, page 643,   
\cite{Bobkov-Ledoux:2017},  
and 
\cite{MR2121458}. 

\section{Questions and extensions:  the bi$-s^*$-concave class}
\label{sec:QuestAndExt}

This immediately raises several questions:  
\begin{description}
\item[Question 1]
What about distributions in classes larger than the log-concave class?
In particular what happens for the $s-$concave classes described by 
\cite{MR0404559}? 
See 
\cite{MR954608},  
and 
\cite{MR0450480}.  
\item[Question 2] 
Is there a class of  bi-$s^*$-concave distributions with the property that if $f$ is $s-$concave, then
$F$ is bi-$s$-concave (or perhaps bi-$s^*$-concave with $s^*$ related to $s$)?
\item[Question 3]
Is there a class of  bi-$s^*$-concave distributions with a theorem analogous to Theorem~\ref{thm:DKWthm} with 
an analogue of Theorem~\ref{thm:DKWthm}(iv)  implying that $\gamma(F) $ is bounded by some function of $s$ 
for all bi-$s^*$-concave distributions $F$?    
\end{description}

We provide positive answers to Questions 1-3 when $s \in (-1,\infty)$, 
beginning with the following definition 
of {\sl bi-$s^*$-concavity} of a distribution function $F$.

\begin{definition}
\label{defn:BiSConcaveDefn}  
For $s \in (-1,\infty]$ we let $s^* \equiv s/(1+s) \in (-\infty, 1]$. 
For $s \in (-1,0)$ we say that a distribution function $F$ on $\RR$ is bi$-s^*-$concave  if both 
$x\mapsto F^{s^*} (x) $ and $x \mapsto (1-F(x))^{s^*}$ are convex functions of $x \in J(F)$.
For $s \in (0,\infty)$ we say that  $F$  is bi$-s^*-$concave  if
$x\mapsto F^{s^*} (x) $ is concave for $x \in (\inf J(F), \infty)$ 
and $x \mapsto (1-F(x))^{s^*}$  is concave for $x \in (-\infty, \sup J(F))$. 
For $s = 0$ we say that $F$ is bi-$0$-concave (or bi-log-concave) if both $x \mapsto \log F(x)$ 
and $x \mapsto \log (1-F(x))$ are concave functions of $x \in J(F)$.  Note that this definition of 
bi-log-concavity 
is equivalent to the definition of bi-log-concavity given by \cite{DuembgenKW:2017}.
\end{definition} 
\smallskip

To briefly explain this definition, recall that a density function $f$ (or just a non-negative function $f$) 
on $\RR$ (or even on $\RR^d$) is $s-$concave 
for $s < 0$ if $f^s$ is convex, while  $f$ is $s-$concave for $s>0$ if $f^s$ is concave on $J(F)$.
Furthermore, from the theory of concave measures due to 
\cite{MR0404559},   
\cite{MR0450480},  
and 
\cite{MR0428540},  
 if $f$ is $s-$concave, the probability measure $P$ on $(\RR,\mathcal{B} )$ defined 
by $P(B) = \int_B f(x) dx$ for Borel sets $B$, is $t - $concave with $t = s/(1+s) \equiv s^*$ if $s>-1$; 
see \cite{MR954608}  
for an introduction, and \cite{MR1898210}  
for a comprehensive review.   
 From  the basic theory of Borell, Brascamp and Lieb, and Rinott,
 it follows easily that if $f$ is $s-$concave with $s \in (-1,\infty]$, then $F$ and $1-F$ are 
 $s^*-$concave; i.e.  the distribution function $F$ corresponding to $f$ is bi$-s^*-$concave.  
 This proof, as well as a simpler calculus type proof assuming that derivatives exist,  
 is given in Section~\ref{sec:SconcavityImpliesBsConcavity}.
 The same argument also establishes the corresponding implication in the log-concave case 
 since, in the log-concave case, $s=0$  and $s^* = 0$ as well.  
 In Section~\ref{sec:Theorem1s} we provide a complete characterization of the class of bi-$s^*$-concave distributions 
 on $\RR$, answering Question 3.  
 
 For the moment we illustrate the definition with several examples.

\begin{example}
\label{exmpl:ex1}
Suppose $f_r$ is the $t-$density with $r >0$ ``degrees of freedom'':
$$
f_r (x) = \frac{C_r}{\left (1 + \frac{x^2}{r} \right )^{(r+1)/2}}  \ \ \ \mbox{for} \ \ x \in \RR.
$$
Here $C_r = \Gamma ((r+1)/2)/(\sqrt{\pi} \Gamma (r/2))$.
It is well-known (see e.g. 
\cite{MR0404559}) 
that $f_r \in \mathcal{P}_s$, the class of $s-$concave densities, if $s \le - 1/(1+r)$.
Note that $s$ takes values in $(-1,0)$ since $r \in (0,\infty)$.
From the Borell-Brascamp-Lieb inequality we guess that the ``right'' transformation $h$ of $F$ and $1-F$ to define
the Bi$-s^*-$concave class is  $h(u) = u^{s^*} = u^{s/(1+s)}$ where $s = -1/(1+r)$, the largest possible value of $s$.  
This leads directly to Definition~\ref{defn:BiSConcaveDefn}.     
Note that $s^*$ in the Borell-Brascamp-Lieb inequality is well-defined since $s>-1$.
Since $s = -(1+r)^{-1}$  we see that we can take $s^* = s/(1+s) = -1/r$ for the $t_r$ family.  Then we want to know if 
$F_r^{-1/r}$ and $(1-F_r)^{-1/r}$ are convex.  Direct computation shows that these are convex functions of $x$.
Plotting these for $r \in \{1/2, 1, 4 \}$ we see that they are indeed convex.   
Moreover we find that 
$\gamma (F_r) = 1 + 1/r = 1/(1+s)$;  this agrees nicely with the log-concave and bi-log-concave picture when $r = \infty$ 
(so $\gamma (F_{\infty} ) = \gamma (N(0,1)) = 1$), and it yields distributions with arbitrarily large values of 
$\gamma (F)$ by considering $\gamma (F_r)$ with $r $ arbitrarily small.  Note, in particular, that this yields 
$\gamma (F_1) = \gamma (Cauchy) = 2$.   Also note that this suggests the conjecture
$\gamma (F) \le 1/(1+s) $ for all  bi-$s^*$-concave distribution functions $F$ where $1/(1+s)$ varies from $1$ to $\infty$ 
as $s$ varies from $0$ to $-1$.
\end{example}
\smallskip
 
\begin{example}
\label{exmpl:ex2}
Suppose that $f_{a,b} $ is the family of $F-$distributions with ``degrees of freedom'' $a>0$ and $b>0$.
(In statistical practice, if $T$ has the density $f_{a,b}$, 
this would usually be denoted by $T\sim F_{b,a}$ where $b$ is the ``numerator degrees of freedom''
and $a$ is the ``denominator degrees of freedom''. )
The density is given by 
\begin{eqnarray*}
f_{a,b} (x) = C_{a,b} \frac{x^{(b/2) -1}}{\left ( a + b x \right )^{(a+b)/2} } \ \ \ \mbox{for} \ \ x \ge 0 .
\end{eqnarray*}
(In fact, $C(a,b) = a^{a/2} b^{b/2} / \mbox{Beta} (a/2,b/2)$, and $f_{a,b} (x) \rightarrow g_b (x)$ as $a\rightarrow \infty$ 
where $g_b$ is the Gamma density with  parameters $b/2$ and $b/2$.)
It is well-known (see e.g. 
\cite{MR0404559}) 
that $f_{a,b} \in \mathcal{P}_s$, the class of $s-$concave densities, 
if $s \le  - 1/(1 + \frac{a}{2} )$ when $a \ge 2$ and $b\ge 2$.   This implies that $s \in [-1/2, 0)$, and the 
resulting $s^* = s/(1+s) $ is in $[-1,0)$.  By Proposition~\ref{prop:1s} it follows that 
$F^{s^*}$ and $(1-F)^{s^*}$ are convex; i.e. $F$ and $1-F$ are $s^*-$ concave.  
This is confirmed by numerical computation.  
\end{example}

\begin{example}
\label{exmpl:ex3} 
Suppose that $f_{a,b} (x) \equiv f(x; a,b) =  (a/b)(x/b)^{-(a+1)} 1_{[b,\infty)} (x) $, the Pareto distribution 
with parameters $a$ and $b$.  In this case $f_{a,b} $ is $s-$concave for each $s \le -1/(1+a)$.  
Thus we take $s = -1/(1+a) \in (-1,0) $ for $a \in (0, \infty)$.  Note that $s^* = s/(1+s) = -1/a$.    
Note that $f^s_{a,b} (x) = (x/b)\cdot (b/a)^{1/(1+a)}$ is certainly convex.  Furthermore, it is easily seen that 
$$
CR_R (x) \equiv (1-F(x)) \frac{f'(x)}{f^2 (x) } = 1 - s^* = 1+1/a \ \ \ \mbox{for all} \ \ x > b.
$$
Thus the Pareto distribution is analogous to the exponential distribution in the log-concave case in the sense that it 
is exactly on the convex {\sl and} concave boundary.
\end{example}

\begin{example}
\label{exmpl:ex4}
Suppose that $f_r (x) = C_r (1 - x^2/r)^{r/2} 1_{[-\sqrt{r} , \sqrt{r}]} (x)$ where $r \in (0,\infty)$.
Here 
$$
C_r =\Gamma ((3+r)/2)/ (\sqrt{\pi r} \Gamma (1+r/2)) .
$$
Note that 
$f_r$ is $s-$concave with $s = 2/r \in (0,\infty)$ since $f_r^{2/r} (x) = C_r^{2/r}(1-x^2/r) 1_{[-\sqrt{r},\sqrt{r}]} (x)$ is concave.
As $r\rightarrow \infty$ it is easily seen that $f_r (x) \rightarrow (2\pi)^{-1/2} \exp (-x^2/2)$, the standard normal density.
Thus $r=\infty$ corresponds to $s=0$.  On the other hand, 
\begin{eqnarray*}
g_r (x) & \equiv &\sqrt{r} f_r (\sqrt{r} x ) = \sqrt{r} C_r (1-x^2)^{r/2} 1_{[-1,1]} (x)\\
& \rightarrow & 2^{-1} 1_{[-1,1]} (x) \ \ \mbox{as} \ r \rightarrow 0.
\end{eqnarray*}
Thus $r=0$ corresponds to $s = + \infty$.
\end{example}

\begin{figure}[ht]
    \centering
    \includegraphics[width=\linewidth,height=5.5cm,keepaspectratio]{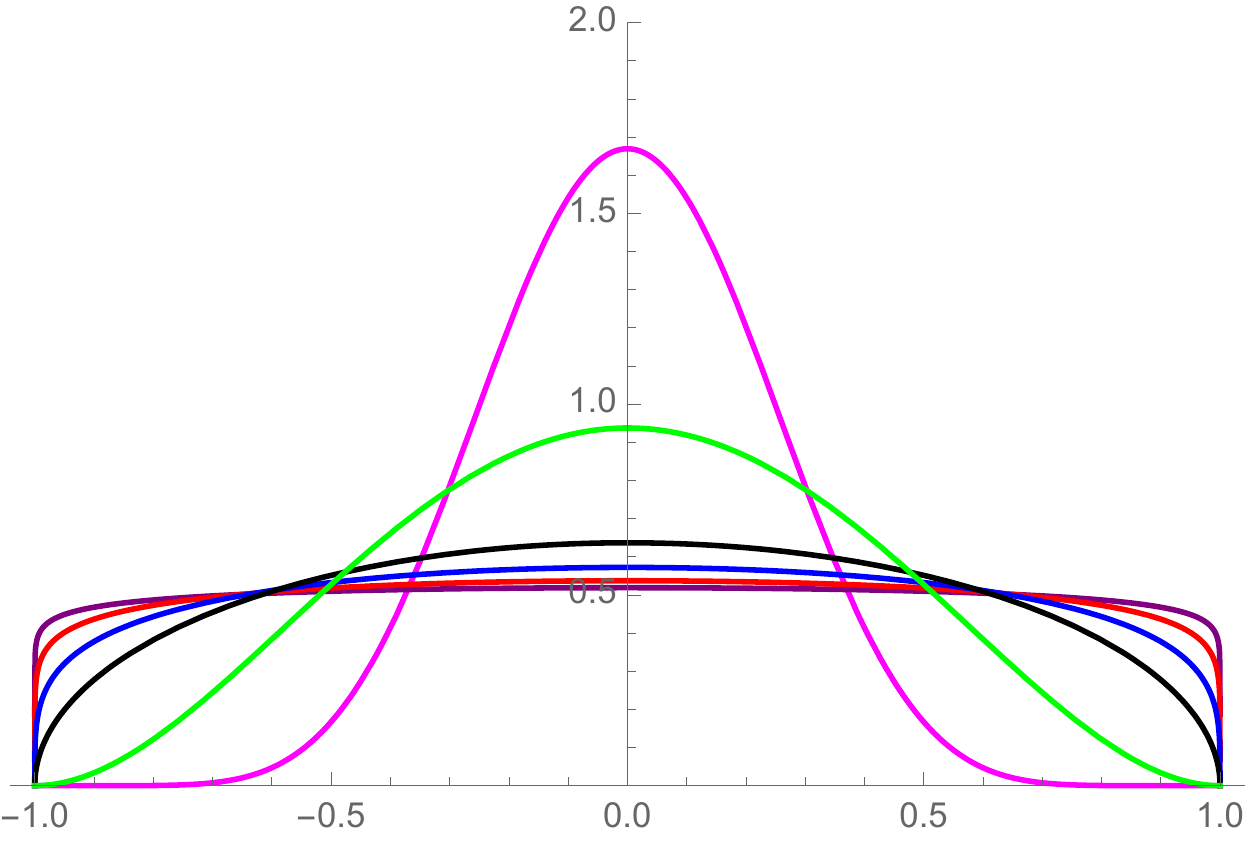}
    \caption{The $s-$concave densities $g_r$ of Example~\ref{exmpl:ex4} with $s = 2/r \in (0,\infty)$: 
    $s=1/8$, magenta; $s=1/2$, green;  $s= 2$, black;  $s=4$, blue;  $s=8$, red;  $s=16$, purple.}
     \label{fig:fig0sConPs}
\end{figure}

\section{$s$-concavity of $f$ implies $s^*$-concavity of $F$ and $1-F$}
\label{sec:SconcavityImpliesBsConcavity}

Motivated by Examples~\ref{exmpl:ex1}-\ref{exmpl:ex4}, 
we first give an extension of the log-concave preservation result of 
\cite{MR2213177}; 
also see Lemma 3 of  \cite{MR1637480}.
\smallskip
 
\begin{proposition}
\label{prop:BagnoliBergstrom} 
(Bagnoli and Bergstrom;  An;  Barlow and Proschan)\\
If $f$ is log-concave then both $F$ and $1-F$ are log-concave; i.e. $\log F$ and $\log (1-F)$ 
are concave.
\end{proposition}

\begin{proposition}
\label{prop:1s}  
If $f$ is $s-$concave with $s \in (-1,\infty)$, 
then both $F $ and $1-F$ are $ s^* = s/(1+s)$ concave; 
i.e. $F^{s^*} $ and $(1-F)^{s^*}$ are convex when $s<0$; and $\log F $ and $\log (1-F)$ are concave
when $s=0$; and  $F^{s^*} $ and $(1-F)^{s^*}$ are convex when $s>0$.  Equivalently, $F$ is bi-$s^*$-concave.
\end{proposition}
  
\begin{remark}
\label{rem:1}
Results related to Proposition~\ref{prop:BagnoliBergstrom} have a long history in reliability theory and econometrics.
\cite{MR0438625} (Lemma 5.8, page 77) showed that if $f$ is log-concave (i.e. $PF_2$, or Polya frequency of order $2$), 
then $f/(1-F)$ is non-decreasing (or ``Increasing Failure Rate'' in their terminology);   they also noted that 
the IFR property is equivalent to $1-F$ being log-concave.  Their proof of the IFR property using the equivalence 
of log-concavity of $f$ and $f \in PF_2$ is delightfully short and does not rely on existence of $f^{\prime}$.
\cite{MR1637480} also proves Proposition~\ref{prop:BagnoliBergstrom} using $PF_2$ 
equivalences to log-concavity without requiring existence of $f^{\prime}$.
The simple ``calculus based'' proof given here and taken from  \cite{MR2213177}, which relies on the 
classical ``second-order conditions'' for convexity (see e.g. \cite{MR2061575}, section 3.1.4), 
was apparently given by \cite{Dierker:1991}, but is likely to have a much longer history.
\end{remark}

In the modern theory of convexity, Proposition~\ref{prop:BagnoliBergstrom} is an immediate consequence of 
the results of \cite{MR0404557}.  
As we will see in the second proof,  Proposition~\ref{prop:1s}  is an immediate consequence 
of the results of \cite{MR0404559}, \cite{MR0450480}, and \cite{MR0428540}.  
\smallskip  

\par\noindent
{\bf Proof of Proposition~\ref{prop:BagnoliBergstrom}} \\
{\bf First Proof, assuming $f^{\prime}$ exists:}\\
Fact 1:  First note that $f$ is log-concave if and only if $f'/f$ is non-increasing. \\
Fact 2:  Note that $F(x) = \int_a^x f(y) dy$ is log-concave if and only if 
$f'(x) F(x) - f^2(x) \le 0$.  To see this, note that 
\begin{eqnarray*}
&& (\log F)^{\prime} (x) = \frac{f}{F} (x), \ \ \mbox{and}\\
&& (\log F)^{\prime \prime} (x) = \frac{f^{\prime \prime}}{F} - \frac{f^2}{F^2} = \frac{f' F - f^2}{F^2} \le 0 .
\end{eqnarray*}
Now if $f$ is log-concave we can use fact 1 to write
\begin{eqnarray*}
\frac{f^{\prime}}{f} (x) F(x) 
& = & \frac{f^{\prime}}{f} (x)\int_a^x f(y) dy \\
& \le & \int_a^x \frac{f^{\prime} (y)}{f(y)} f(y) dy = \int_a^x f^{\prime} (y) dy \\
& = & f(x) - f(a) = f(x) .
\end{eqnarray*}
Rearranging this inequality yields 
$f^{\prime} (x) F(x) - f^2(x) \le 0$, and by Fact 2 we conclude that $F$ is log-concave.
Note that this inequality also can be rewritten as 
$\frac{f^{\prime} (x)}{f^2 (x)} F(x) \le 1$,
and hence we conclude that 
\begin{eqnarray*}
\frac{f^{\prime} (x)}{f^2 (x)} F(x)(1-F(x)) \le 1-F(x) \le 1
\end{eqnarray*}
The argument for $1-F$ is analogous and yields the inequality 
$\frac{f^{\prime} (x)}{f^2 (x)} (1-F(x)) \ge -1$,
and hence we conclude that 
\begin{eqnarray*}
\frac{f^{\prime} (x)}{f^2 (x)} F(x)(1-F(x)) \ge -F(x) \ge -1
\end{eqnarray*}
Thus both $F$ and $1-F$ are log-concave, and $\gamma (F) \le 1$.
\hfill $\Box$
\medskip

\par\noindent
{\bf Second Proof, general (without assuming $f^{\prime}$ exists):}   See the second proof of Proposition~\ref{prop:1s} below.
\medskip

\par\noindent
{\bf Proof of Proposition~\ref{prop:1s}} \\
{\bf First Proof, assuming $f^{\prime}$ exists:}  \\
Suppose $s\in(-1,0)$; the proof for $s>0$ is similar. \\
Fact 1-s:  First note that $f$ is $s-$concave for $s<0$ if and only if $\varphi \equiv f^{s}$ is convex on $J(F)$, which is 
equivalent to $\varphi^{\prime}$  being non-decreasing. But we find
$$
\varphi^{\prime} (x) = (f^s)^{\prime} (x) = s f^{s-1} (x) f^{\prime} (x) = s f^s (x) (f^{\prime} (x) / f(x) ) .
$$
Fact 2-s:  Note that $F(x) = \int_a^x f(y) dy$ is $s^*-$concave for $\st<0$ if and only if 
$$
(s^*-1) f^2 + F f^{\prime}   \le 0\ \text{on }J(F). 
$$  
To see this, note that for $x\in J(F)$
\begin{eqnarray*}
( F^{s^*})^{\prime} (x) & = & s^* F^{s^*-1}(x) f(x) , \ \ \mbox{and}\\
(F^{s^*})^{\prime \prime} (x) 
& = & s^* (s^*-1) F^{s^*} (x)  \left ( \frac{f}{F} (x) \right )^2 + s^* \frac{f^{\prime}(x)}{F(x)} F^{s^*} (x) \\
& = & s^* \frac{F^{s^*}(x)}{F^2 (x)} \left \{ (s^*-1) f^2 (x) + F(x) f^{\prime} (x) \right \} \\
& \ge & 0
\end{eqnarray*}
if and only if (since $s^* = s/(1+s) < 0$)
$$
(s^*-1) f^2 (x) + F(x) f^{\prime}(x) \le 0.
$$
Now if $f$ is $s-$concave and $x\in J(F)$ we can use fact 1-s to write
\begin{eqnarray*}
s f^s (x) \frac{f^{\prime}}{f} (x) F(x) 
& = & s f^s (x) \frac{f^{\prime}}{f} (x)\int_a^x f(y) dy \\
& \ge & \int_a^x s f^s (y) \frac{f^{\prime} (y)}{f(y)} f(y) dy = \int_a^x s f^s (y) f^{\prime} (y) dy \\
& = & \frac{s}{s+1} \left ( f^{s+1} (x) - f^{s+1}(a)\right ) = \frac{s}{s+1} f^{s+1} (x) .
\end{eqnarray*}
Rearranging this inequality (and noting that $s<0$) yields 
$(s^*-1) f^2 + F f^{\prime}   \le 0$, 
and by Fact 2-s we conclude that $F$ is $s^*-$concave.
Note that for $x \in J(F)$ this inequality can also be rewritten as 
$\frac{f^{\prime} (x)}{f^2 (x)} F(x) \le \frac{1}{1+s}$,
and hence we conclude that 
\begin{eqnarray*}
\frac{f^{\prime} (x)}{f^2 (x)} F(x)(1-F(x)) \le \frac{1}{1+s} (1-F(x)) \le \frac{1}{1+s}  = 1 - s^* .
\end{eqnarray*}
The argument for $1-F$ is analogous and yields the inequality 
$\frac{f^{\prime} (x)}{f^2 (x)} (1-F(x)) \ge - \frac{1}{1+s} = - (1-s^*)$,
and hence we conclude that 
\begin{eqnarray*}
\frac{f^{\prime} (x)}{f^2 (x)} F(x)(1-F(x)) \ge - \frac{1}{1+s} F(x) \ge - \frac{1}{1+s}
\end{eqnarray*}
Thus both $F$ and $1-F$ are $s^*-$concave, and $\gamma (F) \le 1/(1+s)$. 
\hfill $\Box$
\medskip

\par\noindent
{\bf Proof of Proposition~\ref{prop:1s}} \\
{\bf Second Proof, general (without assuming $f^{\prime}$ exists):} 
First some background and definitions:\\
$\bullet$ \ \ Let $a,b \ge 0$ and $\theta \in (0,1)$.  
The generalized mean of order $s \in \RR$ is defined by 
\begin{eqnarray*}
M_s (a,b; \theta) 
= \left \{ \begin{array}{l l} ((1-\theta)a^s + \theta b^s )^{1/s},  &  \mbox{if} \ \pm s \in (0,\infty), \\
                                         a^{1-\theta} b^{\theta} , & \mbox{if} \ s = 0,\\
                                         \max\{ a,b \} , & \mbox{if} \ s = \infty, \\
                                         \min\{ a,b \} , & \mbox{if} \ s = - \infty .
              \end{array}
\right .
\end{eqnarray*}
$\bullet$  \ \ Let $(M,d)$ be a metric space with Borel $\sigma-$field $\mathcal{M}$.  
A measure $\mu$ on $\mathcal{M}$ is called {\sl $t-$concave} if for nonempty sets $A,B \in \mathcal{M}$ and $0 < \theta < 1$ we 
have 
$$
\mu_{*} ((1-\theta ) A + \theta B) \ge M_t ( \mu_{*} (A), \mu_{*} (B) ; \theta ) .
$$
$\bullet $ \ \ A non-negative real-valued function $h$ on $(M,d)$ is called {\sl $s-$concave} if for $x,y \in M$ and $0 < \theta <1$ 
we have 
$$
h((1-\theta)x + \theta y) \ge M_s ( h(x), h(y) ; \theta ).
$$
$\bullet$ \ \ Suppose $(M, d) = ( \RR^k, | \cdot | )$, $k-$dimensional Euclidean space with the usual Euclidean metric 
and suppose that $f$ is an $s-$concave density function with respect to Lebesgue measure $\lambda $ on $\mathcal{B}_k$, 
and consider the probability measure $\mu$ on $\mathcal{B}_k$ defined by 
$$
\mu (B) = \int_B f d \lambda  \ \ \ \mbox{for all} \ \ B \in \mathcal{B}_k .
$$
Then by a theorem of Borell (1975), Brascamp and Lieb (1976), and Rinott (1976),
the measure $\mu$ is $s^*$ concave where $s^* = s/(1+ks)$ if $s \in (-1/k,\infty)$ and $s^* = 0$ if $s= 0$.\\
$\bullet$ \ \ Here we are in the case $k=1$.    Thus for $s \in (-1,\infty)$ the measure $\mu$ 
is $s^*$ concave:  for $s \in (-1,\infty)$, $A, B \in \mathcal{B}_1$, and $0 < \theta < 1$,
\begin{eqnarray}
\mu_{*} ( (1-\theta )A + \theta B) \ge M_{s^*} ( \mu_{*} (A) , \mu_{*} (B) ; \theta ) ;
\label{GeneralS-concaveMeasInequalityForR}
\end{eqnarray}
here $\mu_{*}$ denotes inner measure 
(which is needed in general in view of examples noted by \cite{MR0260958}).
With this preparation we can give our second proof of Proposition~\ref{prop:1s}:     
if $A = (-\infty, x]$ and $B = (-\infty, y]$ for $x,y \in J(F)$,
it is easily seen that 
\begin{eqnarray*} 
(1-\theta) A + \theta B 
& = & \{ (1-\theta ) x' + \theta y' : \ x' \le x , \ y' \le y \}\\
& \subset & \{ (1-\theta ) x' + \theta y' : \ (1-\theta )x' + \theta y'  \le (1-\theta) x + \theta y \} \\
& = & (-\infty, (1-\theta)x + \theta y ].
\end{eqnarray*}
Therefore, with the second inequality following from (\ref{GeneralS-concaveMeasInequalityForR})
\begin{eqnarray*}
F((1-\theta) x + \theta y) 
& = & \mu ((-\infty, (1-\theta)x + \theta y])\\
& \ge & \mu ( (1-\theta ) (-\infty,x] + \theta (-\infty, y]) \\
& \ge & M_{s^*} ( \mu ((-\infty,x]) , \mu ((-\infty, y]); \theta ) = M_{s^*} (F(x), F(y); \theta ) ;
\end{eqnarray*}
i.e. $F$ is $s^*-$concave.  Similarly, taking $A = (x,\infty)$ and $B = (y,\infty) $ it follows that
$1-F$ is $s^*-$concave.  

Note that this argument contains a second proof of Proposition~\ref{prop:BagnoliBergstrom} when $s=0$.  
\hfill $\Box$

\section{Bi-$s^*$-concave is (much!) bigger than $s-$concave}
\label{sec:BiIsBigger}

Here we note that just as the class of bi-log-concave distributions is considerably larger 
than the class of log-concave distributions (as shown by \cite{DuembgenKW:2017}), the class of bi$-s^*-$concave distributions is 
considerably larger than the class of $s-$concave distributions.  In particular, multimodal distributions
are allowed in both the bi-log-concave and the bi-s-concave classes. 
 
\begin{example}
\label{exmpl:ex5} 
(\cite{DuembgenKW:2017}, pages 2-3) Suppose that $f$ is  the mixture  $(1/2)N(-\delta,1) +(1/2) N(\delta,1)$.  
\cite{DuembgenKW:2017} showed (numerically) that the corresponding distribution function $F$ is bi-log-concave for $\delta \le 1.34$
but not for $\delta \ge 1.35$.  This distribution has a bi-modal density for $\delta = 1.34$.  
\end{example}  
 
\begin{example}
\label{exmpl:5s}
Now suppose that $f$ is the mixture $(1/2) t_1 (\cdot - \delta) + (1/2) t_1 (\cdot + \delta ) $ with $\delta >0$ 
where $t_r$ is the standard $t$ density with $r$ degrees of freedom as in Example~\ref{exmpl:ex5}.  
By numerical calculation, this density is bi$-s^*-$concave for $\delta = 1.4$, but fails to be bi$-s^*-$concave for $\delta = 1.5$.
Again by numerical calculations the $t_1$ mixture density with  $\delta =1.475 $ is bi-$(-1/2)^*$-concave, 
but with $\delta =1.48$ it is {\sl not} bi-$(-1/2)^*$-concave;  see Figure~\ref{fig:fig4Mixed}.

\end{example}
The following plots illustrate the bounds in Section~\ref{sec:Theorem1s}.

\begin{figure}[ht]
    \centering
    \includegraphics[width=\linewidth,height=5.5cm,keepaspectratio]{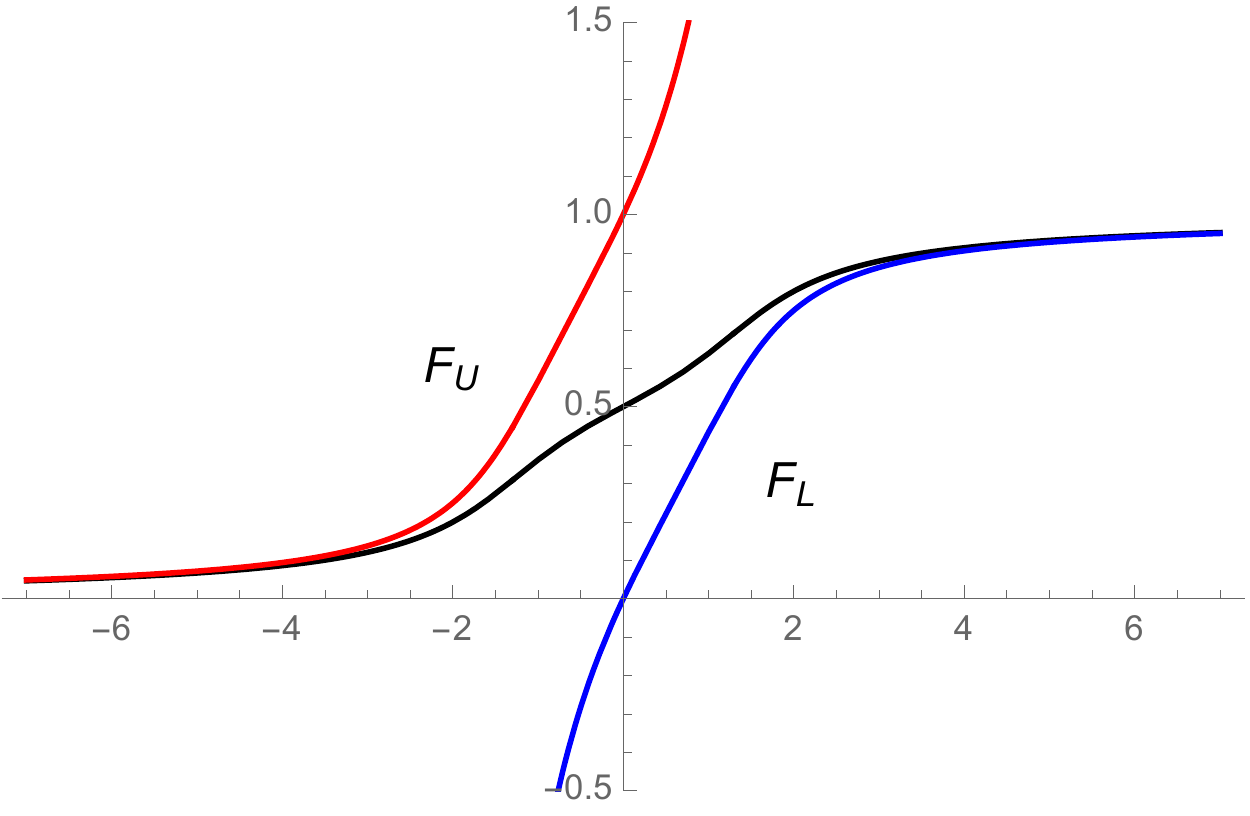}
    \caption{The bi$-s^*-$concave $t_1$ mixture distribution function $F$ (black) for $\delta = 1.3$
    with its convex upper bound $F_U$ (red) and concave lower bound $F_L$ (blue) defined by (\ref{FUpperSNeg}) and (\ref{FLowerSNeg}).}
     \label{fig:fig1Mixed}
 \end{figure}
Upper and lower bounds for the density $f = F^{\prime} $ of $F$ follow from (iii)  of Theorem~\ref{thm:1s}.
These bounds are illustrated for the bi$-s^*-$concave distribution $t_1$ mixture with $\delta = 1.3$ 
in Figure~\ref{fig:fig2Mixed}.
 
 \begin{figure}[ht]
    \centering
    \includegraphics[width=\linewidth,height=5.5cm,keepaspectratio]{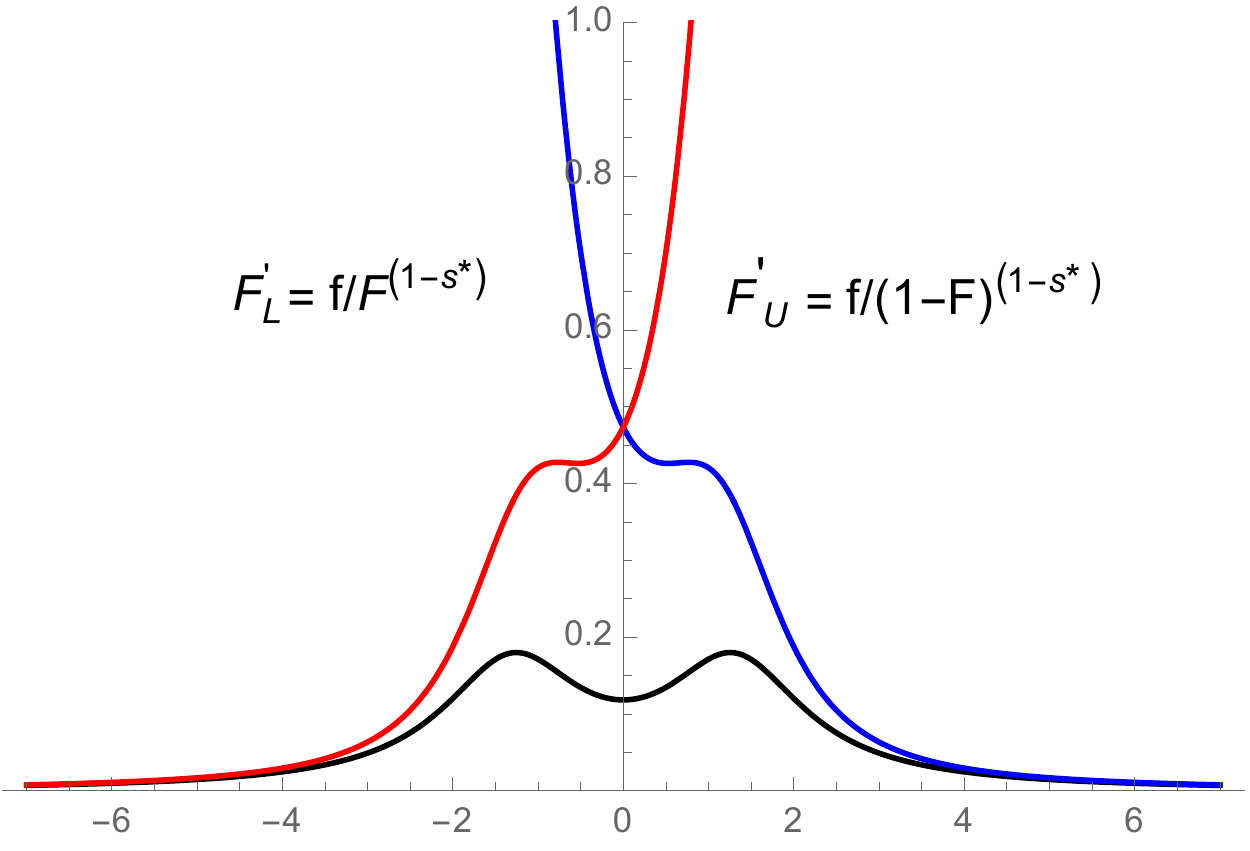} 
    \caption{The bi-$s^*$-concave $t_1$ mixture density function $f$ (black), $\delta = 1.3$, 
    with its bi-$s^*$-concave upper bounds 
    $F_U^{\prime}$ (red) and $F_L^{\prime}$ (blue) defined by (\ref{FprimeUpperSNeg}) and (\ref{FprimeLowerSNeg}).}
     \label{fig:fig2Mixed}
 \end{figure}

\begin{figure}[ht]
    \centering
    \includegraphics[width=\linewidth,height=5.5cm,keepaspectratio]{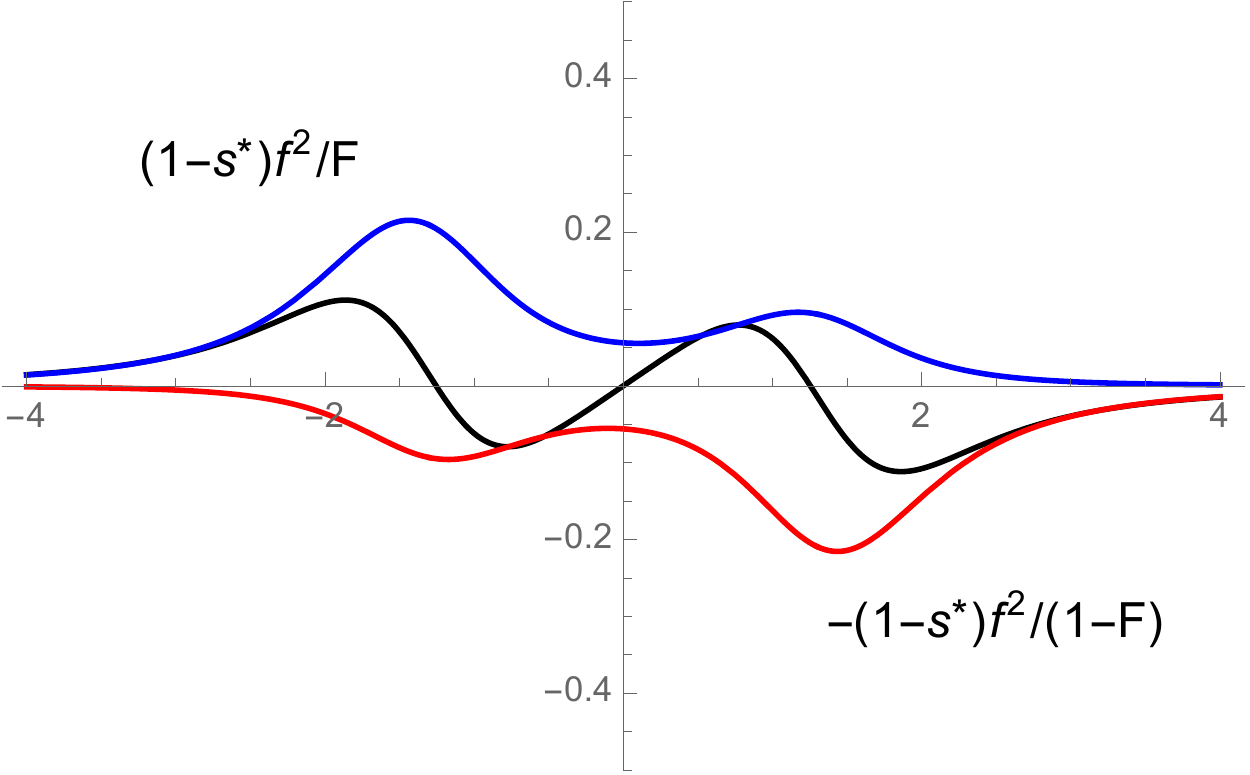} 
    \caption{The bi-$s^*$-concave $t_1$ mixture density function derivative $f^{\prime}$ (black) for $\delta = 1.3$ with its bi-$s^*$-concave 
     upper (blue) and lower (red) bounds as given in (iv) of Theorem~\ref{thm:1s}.}
     \label{fig:fig3Mixed}
 \end{figure}
 
 To get some feeling for what is happening with the Cs\"org\H{o} - R\'ev\'esz condition, 
 Figure~\ref{fig:fig4Mixed} gives plots of  the two functions 
\begin{eqnarray*}
 CR(x) & \equiv & F(x) (1-F(x)) \frac{f^{\prime} (x)}{f^2 (x)} ,\\
 CR_{min} (x) & = & \min\{ F(x), 1-F(x) \} \frac{f^{\prime} (x)} {f^2 (x)} .
 \end{eqnarray*}

\begin{figure}[ht]
    \centering
    \includegraphics[width=\linewidth,height=5.5cm,keepaspectratio]{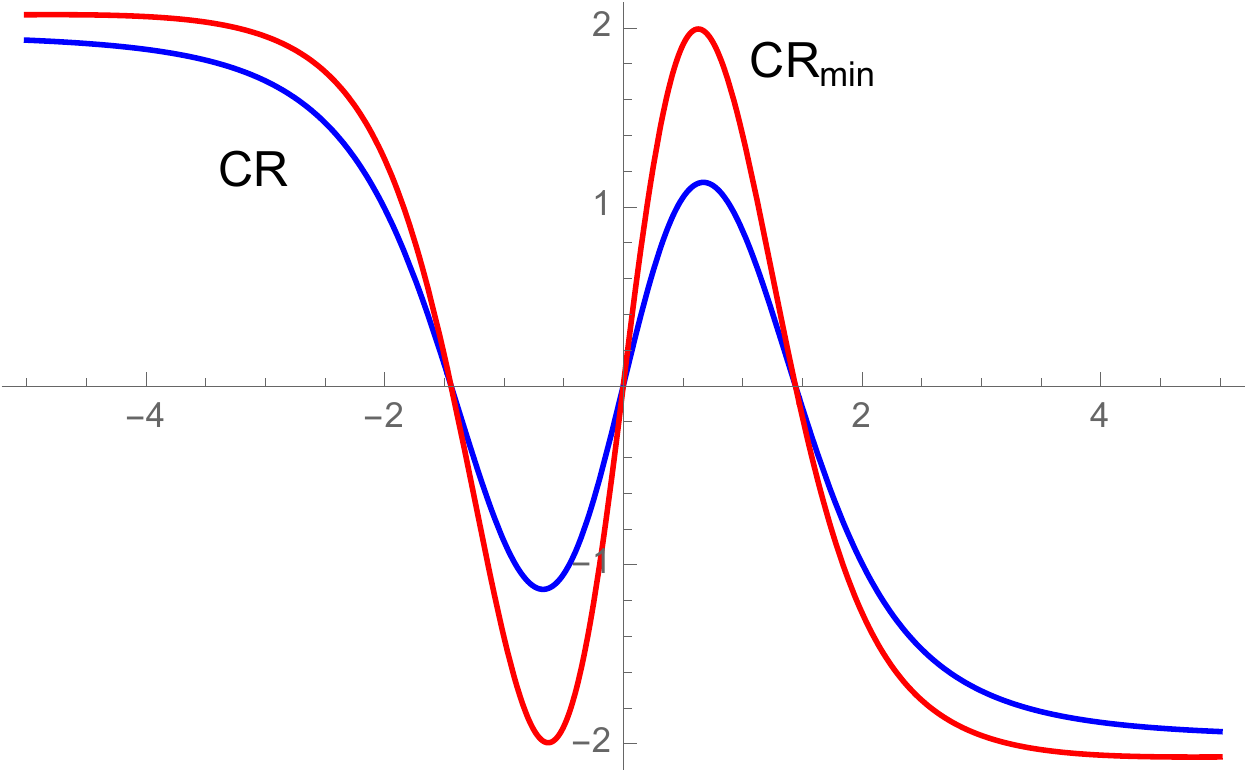} 
    \caption{The Cs\"org\H{o}-R\'ev\'esz functions $CR$ (blue) and  $CR_{min}$ (red) 
    for the mixed $t_1$ density with $\delta = 1.475$ \ .}
     \label{fig:fig4Mixed}
 \end{figure}

 \FloatBarrier
\newpage

\section{The bi-$s^*$-concave analogue of Theorem~\ref{thm:DKWthm}}
\label{sec:Theorem1s}

\subsection{Characterization theorem, bi-$s^*$-concave class}
\label{subsec:CharThm}

Now we can formulate the natural bi-$s^*$-concave analogue of Theorem~\ref{thm:DKWthm}.
   
\begin{theorem}
\label{thm:1s}
Let $s \in (-1,\infty]$. For a non-degenerate distribution function $F$ the following 
four statements are equivalent:\\
(i) \ $F$ is bi-$s^*$-concave.\\
(ii) \ $F$ is continuous on $\RR$ and differentiable on $J(F)$ with derivative $f=F^{\prime}.$ Moreover when $s\leq 0,$ 
\begin{eqnarray}\label{th12}
F(x+t) \left \{ \begin{array}{l}  \le F(x) \cdot \left ( 1+ s^* \frac{f(x)}{F(x)} t \right )_{+}^{1/s^*} \\ 
                \ge 1 - (1-F(x)) \cdot \left (1 - s^* \frac{f(x)}{1-F(x)} t \right )_{+}^{1/s^*} 
\end{array} \right. 
\end{eqnarray}
for all $x\in\RR$ and  $t \in \RR$. When $s>0,$
\begin{align}\label{th12n}
F(x+t)\begin{cases} \ \le F(x) \cdot \left ( 1+ s^* \frac{f(x)}{F(x)} t \right )_{+}^{1/s^*},& \text{ for}\ t\in(a-x,\infty)\\ 
               \ \ge 1 - (1-F(x)) \cdot \left (1 - s^* \frac{f(x)}{1-F(x)} t \right )_{+}^{1/s^*},& \text{ for}\  t\in(-\infty,b-x)
\end{cases}
\end{align}
for all $x \in J(F)$ .\\
(iii) \ $F$ is continuous on $\RR$ and differentiable on $J(F)$ with derivative $f = F^{\prime}$ such that the 
$s^*-$hazard function $ f/(1-F)^{1-s^*}$ is non-decreasing, and the 
reverse $s^*-$hazard function $  f/F^{1-s^*}$ is 
non-increasing on $J(F)$.\\
(iv) \ $F$ is continuous on $\RR$ and differentiable on $J(F)$ with bounded 
and strictly positive derivative $f = F^{\prime}$.  Furthermore, $f$ is locally Lipschitz-continuous on $J(F)$ with 
$L^1-$derivative  $f^{\prime} = F^{\prime \prime} $ satisfying 
\begin{eqnarray}\label{4eq3}
-(1-s^*)\frac{f^2}{1-F} \le f^{\prime} \le (1-s^*)\frac{f^2}{F} .
\end{eqnarray}
Recall that $s^* = s/(1+s) \in (-\infty, 1]$ and  $(1-s^*) = 1/(1+s) \in [0,\infty)$.  
Alternatively,
\begin{eqnarray*}
-\frac{f^2}{1-F} \le (1+s) f^{\prime} \le \frac{f^2}{F} .
\end{eqnarray*}
\end{theorem}
\smallskip

This yields the following corollary extending (\ref{GammaFboundedbyOneForLogConcaveF}) 
from $s=0$ to $s \in (-1,\infty]$.

\begin{corollary}
\label{cor:CR-gammaBound} 
Suppose that $F$ is bi-$s^*$-concave for $s \in (-1, \infty]$.  
Then 
\begin{eqnarray*}
\gamma (F) = \sup_{x \in J(F)} F(x) (1-F(x)) \frac{| f' (x)|}{f^2(x)} & \le & 1- s^* = \frac{1}{1+s} , 
\end{eqnarray*} 
and 
\begin{eqnarray*}
\tilde{\gamma} (F) = \sup_{x \in J(F)} \min\{ F(x) , 1-F(x) \}  \frac{| f' (x)|}{f^2(x)} & \le & 1- s^*  = \frac{1}{1+s} .
\end{eqnarray*}
\end{corollary}
\smallskip
 
\begin{remark}
\label{rem:SW-connection}
The three distribution functions $F$  considered by \cite{MR838963, MR3396731} 
page 644 all involved log-concave densities 
with the resulting bound for $\gamma (F)$ being $1$.  Theorem~\ref{thm:1s} and 
Corollary~\ref{cor:CR-gammaBound} give a rather complete description 
of how the values of $\gamma(F)$ and $\tilde{\gamma} (F)$  depend on the index $s^*$ of bi-$s^*$-concavity.
\end{remark}

\par\noindent
\begin{proof}[Proof of Theorem~\ref{thm:1s}]
If $s=0$, the proof follows from Theorem~\ref{thm:DKWthm} of \cite{DuembgenKW:2017}.  
When $s=\infty$, $s^*=1$ and $1-s^* = 0$.  In this case $f'=0$ almost everywhere (Lebesgue) and 
$f$ is a uniform density on $(a,b)$.  When $s \in (0,\infty)$ the proof is essentially the same as for $s=0$
with only two minor modifications (in the proof of (i) implies (ii) and in the proof of (iii) implies (iv));  
see the Appendix section~\ref{sec:appendix} for complete 
details.
It remains to consider the case 
when $s\in(-1,0)$.  Our proof closely parallels the proof 
for the case $s=0$ given by  \cite{DuembgenKW:2017}.
Throughout our proof we will denote $\inf J(F)$ and $\sup J(F)$ by $a$ and $b$ 
respectively. Notice that if $F$ is continuous, $J(F)=(a,b).$

  Proof of (i) implies (ii):
 Since $F$ is bi-$s^*$-concave with $s^{*}<0$, $\psi=F^{1/\st}$ is convex on $J(F)$. 
 Since $\psi(x)=1$ and $\infty$ for $x\geq\sup J(F)$ and $x\le\inf J(F)$ respectively,  $\psi$
 is convex on $\RR.$ By the convex version of Lemma $6$ of \cite{DuembgenKW:2017} 
  $\psi$ is continuous on the interior of $\{\psi<\infty\}.$ Therefore $\psi$ and hence $F$ is continuous 
  on the interior of the set $\{F>0\}$ or $(a,\infty)$. 
  Similarly, the $s^*$-concavity of $1-F$ 
  implies continuity of $1-F$ on the interior of the set $\{1-F>0\}:=(-\infty,b)$ where 
  $b:=\sup\{F<1\}$.  
  However unless $a<b$, $F$ would be degenerate. Hence, $a<b$ and $F$ is continuous on $\RR$. 
  More precisely $J(F)=(a,b)$.  
  
 Let $x\in(a,b)$. Convexity of $\psi$ implies that 
  \begin{eqnarray*}
  F'(x\pm)=\lim_{t\to 0,\pm t>0}\dfrac{\psi^{1/s^*}(x+t)-\psi^{1/s^*}(t)}{t}=\dfrac{1}{s^*}\psi(x)^{1/s^*-1}\psi'(x\pm)
  \end{eqnarray*}
  exist and satisfy
  \begin{eqnarray*}
  F'(x-) \le F'(x+) .
  \end{eqnarray*}
 Similarly, convexity of $(1-F)^{s^*}$ yields 
  \begin{eqnarray*}
  (1-F)'(x+)\geq (1-F)'(x-)
  \end{eqnarray*}
  which implies that
  \begin{eqnarray*}
  -F'(x+)\geq -F'(x-).
  \end{eqnarray*}
Therefore $F'(x-)=F'(x+)$ which proves the differentiability of $F$. 
It also shows that $\psi'(x+)=\psi'(x-)=\psi'(x)$ on $(a,b)$.
  
By Lemma 6 (convex version) 
    of \cite{DuembgenKW:2017} for each $x\in(a,b)$ and $c\in[\psi'(x-),\psi'(x+)]$ one has
    \begin{eqnarray*}
    \psi(x+t) -\psi(x)\geq ct\ \text{ for all }t\in\RR.
    \end{eqnarray*}
    Therefore
    \begin{eqnarray*}
    \psi(x+t)-\psi(x)\geq t \psi'(x).
    \end{eqnarray*}
    Hence, 
    \begin{eqnarray*}
    F^{s^*}(x+t)-F^{s^*}(x)\geq ts^{*}f(x)F(x)^{s^*-1} ,
    \end{eqnarray*}
    or, with $x_{+}=\max\{x,0\}$,
    \begin{eqnarray*}
    \dfrac{F^{s^*}(x+t)}{F^{s^*}(x)}\geq\bigg ( 1+s^*\dfrac{f(x)}{F(x)}t\bigg )_{+}.
    \end{eqnarray*}
    Hence,
\begin{eqnarray*}
    \dfrac{F(x+t)}{F(x)}\leq\bigg ( 1+s^*\dfrac{f(x)}{F(x)}t\bigg )_{+}^{1/s^*}.
\end{eqnarray*}
     Analogously it follows that
\begin{eqnarray*}
    (1-F(x+t))^{s^*}-(1-F(x))^{s^*} \geq -ts^*f(x)(1-F(x))^{s^*-1}
\end{eqnarray*}
   which yields
\begin{eqnarray*}
    \bigg(\dfrac{1-F(x+t)}{1-F(x)}\bigg )^{s^*}\geq \bigg (1-ts^*\dfrac{f(x)}{1-F(x)}\bigg)_{+}
\end{eqnarray*}
    or
\begin{eqnarray*}
    F(x+t)\geq 1-(1-F(x))\cdot \bigg (1-ts^*\dfrac{f(x)}{1-F(x)}\bigg)_{+}^{1/s^*}.
\end{eqnarray*}
    Hence \eqref{th12} is proved.
    
Since
    (ii) holds, $F$ is continuous and differentiable on $J(F)$ with derivative $f=F'$ 
    and satisfies \eqref{th12}. Now let $x,y\in J(F)$ with $x<y.$  
    Let 
\begin{eqnarray}
    h=f/F^{1-s^*}.
    \label{defh}
\end{eqnarray} 
    Then applying \eqref{th12} we obtain that
\begin{eqnarray*}
    \dfrac{F^{s^*}(x)}{F^{s^*}(y)}\geq 1+s^*\dfrac{f(y)}{F(y)}(x-y).
\end{eqnarray*}
    Hence, 
\begin{eqnarray*}
    F^{s^*}(x)
    & \geq & F^{s^*}(y)+s^*\dfrac{f(y)}{F(y)^{1-s^*}}(x-y)\\
    & = &  F^{s^*}(y)+s^*h(y)(x-y)\\
    &\geq & \ F^{s^*}(x)+s^*h(x)(y-x)+s^*h(y)(x-y).
\end{eqnarray*}
    Therefore 
\begin{eqnarray*}
    s^{*}(x-y)(h(y)-h(x))\leq 0
\end{eqnarray*} 
    where $s^{*}(x-y)>0$,  implying that $h(y)\leq h(x)$. 
    Therefore $h$ is non-increasing. Now let 
\begin{eqnarray}
    \tilde{h}=f/(1-F)^{1-s^{*}}. \label{defth}
\end{eqnarray} 
     From \eqref{th12} we also obtain that 
\begin{eqnarray*}
    (1-F(x))^{s^{*}}-(1-F(y))^{s^*}\geq -ts^{*}\dfrac{f(y)}{(1-F(y))^{1-s^*}}=-ts^{*}\tilde{h}(y)
\end{eqnarray*}
    or
\begin{eqnarray*}
      (1-F(x))^{s^{*}}& \geq & (1-F(y))^{s^*}  -(x-y)s^{*}\tilde{h}(y)\\
         & = & (1-F(x))^{s^*}-(y-x)s^{*}\tilde{h}(x)  -(x-y)s^{*}\tilde{h}(y)\\
         & = & (1-F(x))^{s^*}-s^{*}(y-x)(\tilde{h}(y)-\tilde{h}(x) ).\\
\end{eqnarray*}
    Since $s^{*}(y-x)<0,$ the last inequality leads to
\begin{eqnarray*}
    0\leq \tilde{h}(y)-\tilde{h}(x), 
\end{eqnarray*} 
    implying that $\tilde{h}$ is non-decreasing. 

Proof of (iii) implies (iv): 
   If the conditions of (iii) hold, then it immediately follows that $f>0$ on $J(F)$. 
    If not, suppose that $f(x_0) = 0$ for some $x_0\in J(F)$.  Now $J(F)=(a,b)$ since $F$ is continuous.
    Since $f(x)/F(x)^{1-s^*}$ is non-increasing, 
    $f(x)=0$ for $x\in [x_0,b).$ Similarly since $f(x)/(1-F(x))^{1-s^*}$ is non-decreasing 
    we obtain $f(x)=0$ for $x\in(a, x_0].$ 
    Therefore, $F'=0$  or $F$ is constant on $J(F)$. Then $F$ violates the continuity 
    condition of (iii). Hence $f>0$ on $J(F)$.  
  
Suppose $h$ and $\tilde{h}$ are as defined in \eqref{defh} and \eqref{defth}. 
    Then the monotonicities of $h$ and $\tilde{h}$ imply that for any $x,x_0\in J(F),$
\begin{eqnarray*}
      f(x)=\begin{cases}
      F^{1-s^*}(x)h(x)\leq h(x_0) &  \ \mbox{if} \ \  x\geq x_0,\\
      (1-F(x))^{1-s^*}\tilde{h}(x)\leq \tilde{h}(x_0) & \ \mbox{if} \ \  x\leq x_0.
      \end{cases}
 \end{eqnarray*}
 Next, let $c,d \in J(F)$ with $c<d$. 
 We will  bound $(f(y)-f(x))/(y-x)$ for $x,y\in J(F)$ such that $x,y\in(c,d)$ with $x \neq y$.
  This will yield local Lipschitz-continuity of $f$ on $J(F)$. 
  To this end, note that
\begin{eqnarray*}
 \dfrac{f(y)-f(x)}{y-x}
 & = & \ \dfrac{F^{1-s^*}(y)h(y)-F^{1-s^*}(x)h(x)}{y-x}\\
 & = & \  h(y)\dfrac{F^{1-s^*}(y)-F^{1-s^*}(x)}{y-x}+F^{1-s^*}(x)\dfrac{h(y)-h(x)}{y-x}\\
 & \leq & \  h(c)\dfrac{F^{1-s^*}(y)-F^{1-s^*}(x)}{y-x}\\
 & \to &  \  h(c) (1-s^*)f(x) F^{-s^*}(x)=(1-s^{*})h(c)h(x)F^{1-2s^{*}}(x)
\end{eqnarray*} 
as $y\to x$.  Here the inequality followed from the fact that 
\begin{eqnarray*}
\dfrac{h(y)-h(x)}{y-x}\leq 0
\end{eqnarray*} 
which holds since $h$ is non-increasing.
Now  since $h(x)\leq h(c)$, $1-2 s^{*}>0$, $1-s^{*} >0$, and $F(x)\leq F(d)$, we find that
\begin{eqnarray*}
\limsup_{y\to x}\dfrac{f(y)-f(x)}{y-x}\leq (1-s^*)h(c)^2F^{1-2s^*}(d)
\label{4eq1}
\end{eqnarray*}
 for all $x\in(c,d)$.
Analogously with $\bar{F}=1-F$ we obtain that
\begin{eqnarray*}
\dfrac{f(y)-f(x)}{y-x}
& = & \ \dfrac{\bar{F}^{1-s^*}(y)\tilde{h}(y)-\bar{F}^{1-s^*}(x)\tilde{h}(x)}{y-x}\\
& = & \tilde{h}(y)\dfrac{\bar{F}^{1-s^*}(y)-\bar{F}^{1-s^*}(x)}{y-x}+\bar{F}^{1-s^*}(x)\dfrac{\tilde{h}(y)-\tilde{h}(x)}{y-x}\\
& \geq & \  \tilde{h}(y)\dfrac{\bar{F}^{1-s^*}(y)-\bar{F}^{1-s^*}(x)}{y-x}
\end{eqnarray*} 
since, by the non-decreasing property of $\tilde{h}$, for any $x,y\in J(F)$, 
\begin{eqnarray*}
\dfrac{\tilde{h}(y)-\tilde{h}(x)}{y-x}>0.
\end{eqnarray*} 
Next observe that since $1-s^*=1/(1+s)>0$, and $\bar{F}$ is nonincreasing, 
\begin{eqnarray*}
\tilde{h}(y)\dfrac{\bar{F}^{1-s^*}(y)-\bar{F}^{1-s^*}(x)}{y-x}\geq \tilde{h}(d)\dfrac{\bar{F}^{1-s^*}(y)-\bar{F}^{1-s^*}(x)}{y-x}.
\end{eqnarray*}
Hence as $y\to x$ it follows that   
\begin{eqnarray*}
\liminf_{y\to x}\dfrac{f(y)-f(x)}{y-x}
& \geq &-\tilde{h}(d)(1-s^*)f(x)\bar{F}^{-s^*}(x)\\
& = & -\tilde{h}(d)\tilde{h}(x)(1-s^*)\bar{F}^{1-2s^*}(x).
\end{eqnarray*}
Therefore using the fact that $\tilde{h}(x)\leq \tilde{h}(d)$ and $1-2s^*>0$ we conclude that 
\begin{eqnarray*}
\liminf_{y\to x}\dfrac{f(y)-f(x)}{y-x}\geq -\tilde{h}(d)^2(1-s^*)\bar{F}^{1-2s^*}(c).
\label{4eq2}
\end{eqnarray*}
Combining the above with \eqref{4eq1} we find that $f$ is Lipschitz-continuous on $(c,d)$ with Lipschitz-constant 
\begin{eqnarray*}
\max\{(1-s^*)h(c)^2F^{1-2s^*}(d),(1-s^*)\tilde{h}(d)^2\bar{F}^{1-2s^*}(c)\}.
\end{eqnarray*} 
This proves that $f$ is locally Lipschitz continuous on $J(F)$. 
Hence, $f$ is also locally absolutely continuous with $L^1$-derivative $f'$ such that
\begin{eqnarray*}
f(y)-f(x)=\int_{x}^{y}f'(t)dt\ \text{  for all }x,y \in J(F);
\end{eqnarray*}
hence $f'(x)$ can be chosen so that
\begin{eqnarray*}
f'(x)\in\bigg[\liminf_{y\to x}\dfrac{f(y)-f(x)}{y-x},\limsup_{y\to x}\dfrac{f(y)-f(x)}{y-x}\bigg].
\end{eqnarray*}
However \eqref{4eq1} and \eqref{4eq2} imply that for $c<x<d$,
\begin{eqnarray*}
\lefteqn{
\bigg[\liminf_{y\to x}\dfrac{f(y)-f(x)}{y-x},\limsup_{y\to x}\dfrac{f(y)-f(x)}{y-x}\bigg]}\\
&& \subset \bigg[-(1-s^*)\tilde{h}(d)^2\bar{F}^{1-2s^*}(c),(1-s^*)h(c)^2F^{1-2s^*}(d)\bigg]
\end{eqnarray*}
Now since $f$ and $F$ are continuous and $F>0$ on $J(F),$ so are $h$ and $\tilde{h}$. 
Therefore,  letting $c,d\to x$ it follows that 
\begin{eqnarray*}
\dfrac{-(1-s^*)f(x)^2\bar{F}^{1-2s^*}(x)}{\bar{F}^{2-2s^*}(x)}\leq f'(x)\leq (1-s^*)\dfrac{f(x)^2F^{1-2s^*}(x)}{F^{2-2s^*}(x)};
\end{eqnarray*}  
and this implies \eqref{4eq3}.
          
Proof of (iv) implies (i):
 The fact that (iii) implies (i) can be easily proved  since $f/F^{1-s^*}$ non-increasing on 
 $J(F)$ implies that   $F^{s^*}$ is convex on $J(F).$ Also $1< F^{s^*}<\infty$ on $J(F)$.
 Now $F^{s^*}(x)=\infty$ for $x< \inf J(F)$ and $F^{s^*}(x)=1$ for $x> \sup J(F)$. 
 Therefore $F^{s^*}$ is convex on $\RR$.
 Similarly one can show that $(1-F)^{s^*}$ is convex on $\RR$.
 Hence $F$ is bi-$s^*$-concave. Therefore it is enough to prove that (iv) implies (iii).
 
 By Lemma $7$ of \cite{DuembgenKW:2017} $h$ is non-increasing on $J(F)$ 
 if and only if for any $x\in J(F)$ the following holds:
 \begin{eqnarray*}
 \limsup_{y\to x}\dfrac{h(y)-h(x)}{y-x}\leq 0.
 \end{eqnarray*}
 Suppose $x\neq y\in J(F)$ and $r:=\min(x,y)$ and $s:=\max(x,y).$ Then it follows that
 \begin{eqnarray*}
 \lefteqn{\dfrac{h(y)-h(x)}{y-x} =  \dfrac{f(y)/F^{1-s^*}(y)-f(x)/F^{1-s^*}(x)}{y-x}}\\
 && =\  \dfrac{1}{F^{1-s^*}(y)}\dfrac{f(y)-f(x)}{y-x}-\dfrac{f(x)}{F^{1-s^*}(x)F^{1-s^*}(y)}\dfrac{F^{1-s^*}(y)-F^{1-s^*}(x)}{y-x}\\
 && = \  \dfrac{1}{F^{1-s^*}(y)}\dfrac{\int_{r}^{s}f'(t)dt}{s-r}-\dfrac{f(x)}{F^{1-s^*}(x)F^{1-s^*}(y)}\dfrac{F^{1-s^*}(y)-F^{1-s^*}(x)}{y-x}\\
 && \leq \  \dfrac{(1-s^*)}{F^{1-s^*}(y)(s-r)}\int_{r}^{s}\dfrac{f(t)^2}{F(t)}dt-\dfrac{f(x)}{F^{1-s^*}(x)F^{1-s^*}(y)}\dfrac{F^{1-s^*}(y)-F^{1-s^*}(x)}{y-x}.
 \end{eqnarray*}
  by \eqref{4eq3}. 
  Since $F$ is continuous by (iv), $J(F)$ must be an interval. 
  Also since $x,y\in J(F),$ $[r,s]\subset J(F)$.  Since $f$ and $F$ are continuous on 
  $J(F)$ and $F>0$ on $J(F)$,  $f^2/F$ is continuous and integrable on $J(F)$ and hence also on $[r,s]$.    
  Letting $y\to x$ we obtain that
 \begin{eqnarray*}
 \limsup_{y\to x}\dfrac{h(y)-h(x)}{y-x}\leq \dfrac{(1-s^*)f(x)^2}{F^{2-s^*}(x)}-\dfrac{(1-s^*)f(x)^2}{F^{2-s^*}(x)}=0.
 \end{eqnarray*}
 Analogously by Lemma $7$ of \cite{DuembgenKW:2017}, to show $\tilde{h}$ is non-decreasing it is enough to show that
 \begin{eqnarray*}
 \liminf_{y\to x}\dfrac{\tilde{h}(y)-\tilde{h}(x)}{y-x}\geq 0.
 \end{eqnarray*}
 To verify this suppose $x\neq y\in J(F)$ and $r:=\min(x,y)$ and $s:=\max(x,y)$. 
 As before we calculate
  \begin{eqnarray*}
 \dfrac{\tilde{h}(y)-\tilde{h}(x)}{y-x}
 & = &\dfrac{f(y)/\bar{F}^{1-s^*}(y)-f(x)/\bar{F}^{1-s^*}(x)}{y-x}\\
 & = & \dfrac{1}{F^{1-s^*}(y)}\dfrac{f(y)-f(x)}{y-x}-\dfrac{f(x)}{\bar{F}^{1-s^*}(x)\bar{F}^{1-s^*}(y)}\dfrac{\bar{F}^{1-s^*}(y)-\bar{F}^{1-s^*}(x)}{y-x}\\
 & = &  \dfrac{1}{\bar{F}^{1-s^*}(y)}\dfrac{\int_{r}^{s}f'(t)dt}{s-r}-\dfrac{f(x)}{\bar{F}^{1-s^*}(x)\bar{F}^{1-s^*}(y)}\dfrac{\bar{F}^{1-s^*}(y)-\bar{F}^{1-s^*}(x)}{y-x}\\
 & \geq &   - \ \dfrac{(1-s^*)}{\bar{F}^{1-s^*}(y)(s-r)}\int_{r}^{s}\dfrac{f(t)^2}{\bar{F}(t)}dt-\dfrac{f(x)}{\bar{F}^{1-s^*}(x)\bar{F}^{1-s^*}(y)}\dfrac{\bar{F}^{1-s^*}(y)-\bar{F}^{1-s^*}(x)}{y-x}
 \end{eqnarray*}
 by \eqref{4eq3}. Since $f$ and $\bar{F}$ are continuous on 
 $J(F)$,  letting $y\to x$ it follows that
 \begin{eqnarray*}
 \liminf_{y\to x}\dfrac{\tilde{h}(y)-\tilde{h}(x)}{y-x} 
 \geq  -\dfrac{(1-s^*)f(x)^2}{\bar{F}^{2-s^*}(x)}+\dfrac{(1-s^*)f(x)^2}{\bar{F}^{2-s^*}(x)}=0.
 \end{eqnarray*}
\end{proof}
\smallskip

\subsection{Bounds for $F$ bi-$s^*$-concave when $s<0$.}
\label{subsec:boundsSNeg}  

First, upper and lower bounds on $F$:   Note that $(1+y)^r \ge 1+ ry$ for any $r < 0$ and $y \ge -1$.
Taking $y = -F(x)$ and $r = s^*$ yields 
\begin{eqnarray*}
(1-F(x))^{s^*} \ge 1- s^* F(x) 
\end{eqnarray*}
or, by rearranging,
\begin{eqnarray}
F(x) \le \frac{1}{-s^*} \left \{ (1-F(x))^{s^*} - 1 \right \} \equiv F_{U,s} (x) \equiv F_U (x)
\label{FUpperSNeg}
\end{eqnarray}
where $F_U$ is a convex function if $F$ is bi$-s^*-$concave.  
Similarly, taking $y= - (1-F(x))$ and $r=s^*$ yields, by rearranging terms
\begin{eqnarray}
F(x) \ge \frac{1}{-s^*} \left \{ (1-s^*) - F(x)^{s^*} \right \} \equiv F_{L,s} (x) \equiv F_L (x) 
\label{FLowerSNeg}
\end{eqnarray}
where $F_L$ is a concave function if $F$ is bi$-s^*-$concave.  
Note that 
\begin{eqnarray}
F_U^{\prime} (x) =  \frac{f(x)}{(1-F(x))^{1-s^*}}  = \frac{f(x)}{(1-F(x))^{1/(1+s)}} 
\label{FprimeLowerSNeg}
\end{eqnarray}
is monotone non-decreasing, while
\begin{eqnarray}
 F_L^{\prime} (x)  =  \frac{f(x)}{F^{1-s^*}(x)}  = \frac{f(x)}{F^{1/(1+s)} (x)}  
\label{FprimeUpperSNeg}
\end{eqnarray} 
is monotone non-increasing.
Therefore 
\begin{eqnarray*}
&& 0 \le F_U^{\prime \prime} (x) =  (1-F(x))^{s^*-2} \left \{ (1-s^*) f^2 (x) + (1-F(x)) f^{\prime} (x) \right \}, \\
&& 0 \ge F_L^{\prime \prime} (x)  =   F(x)^{s^*-2} \left \{ (s^*-1) f^2 (x) + F(x) f^{\prime}(x) \right \} . 
\end{eqnarray*}
The upper and lower bounds in (iv) of Theorem~\ref{thm:1s} follow by rearranging these inequalities.

Taking $F$ to be the distribution function of  $t_1$ and plotting the bounds  for $F$, $F^{\prime} = f$ 
and $F^{\prime \prime} = f^{\prime}$ yields the following three figures.  

\begin{figure}[ht]
    \centering
    \includegraphics[width=\linewidth,height=5.5cm,keepaspectratio]{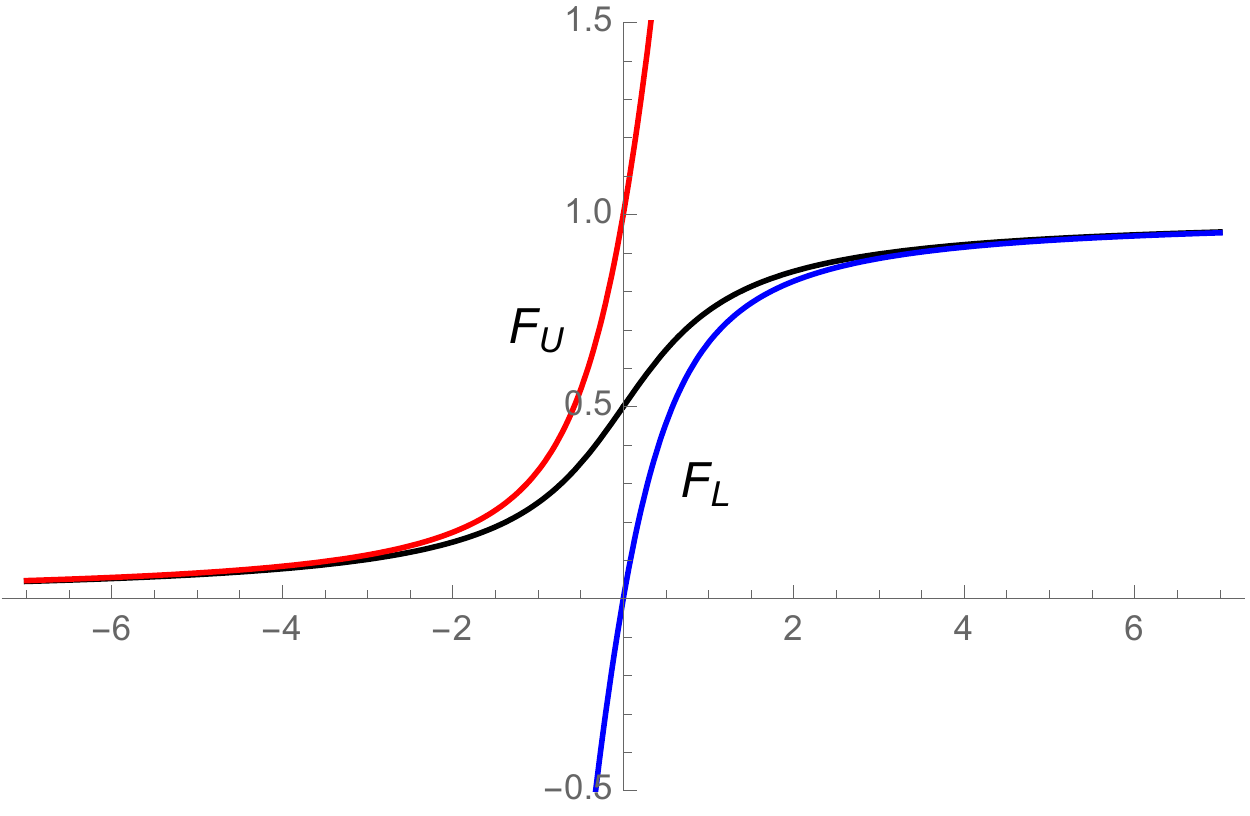}
    \caption{The bi$-s^*-$concave $t_1$ distribution function $F$ (black) with its  
    convex upper bound $F_U$ (red) and concave lower bound $F_L$ (blue), 
    where $F_U$ and $F_L$ are given in (\ref{FUpperSNeg}) and (\ref{FLowerSNeg}).}
    \label{fig:fig1}
 \end{figure}

Upper  bounds for the density $f = F^{\prime} $ of $F$ follow from (iii)  of Theorem~\ref{thm:1s}: 
These bounds are illustrated for the bi-$s^*$-concave distribution $t_1$ in Figure~\ref{fig:fig2}.
 
  \begin{figure}[ht]
    \centering
    \includegraphics[width=\linewidth,height=5.5cm,keepaspectratio]{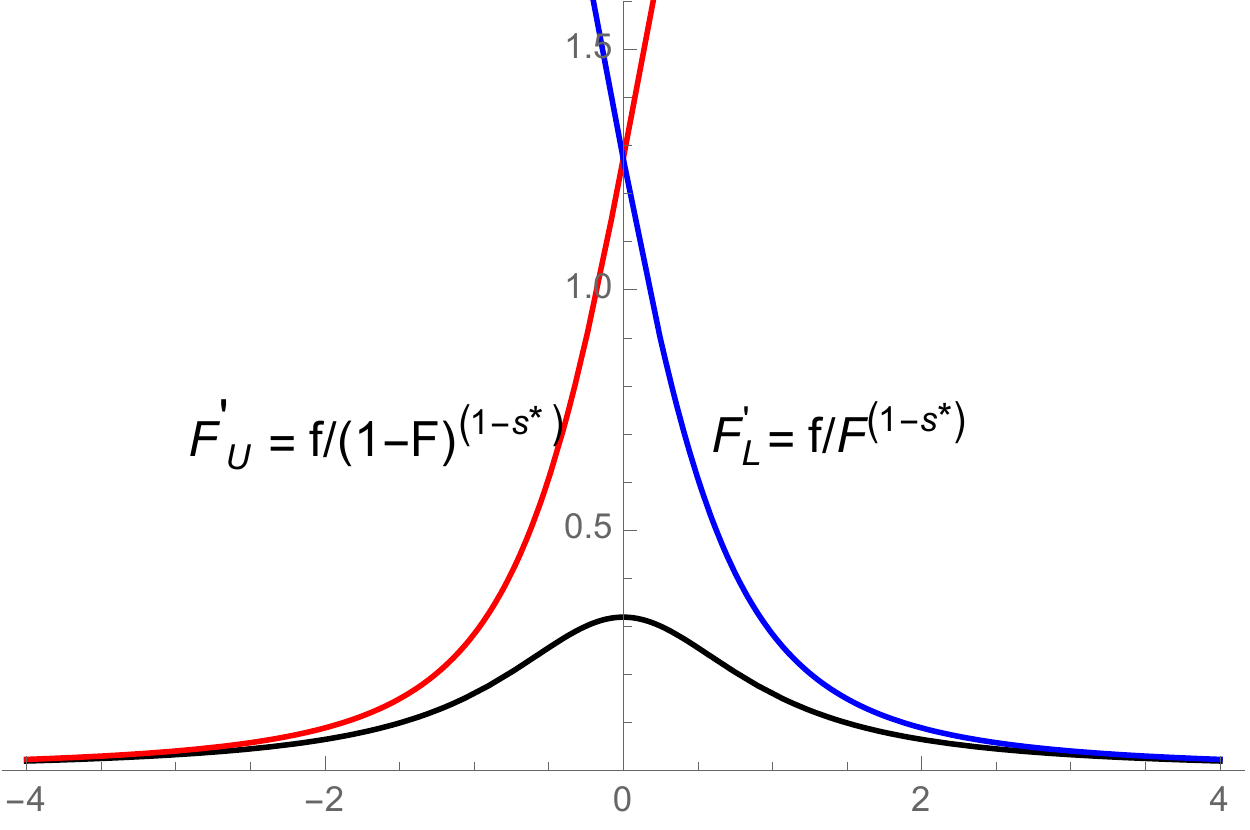} 
    \caption{The bi-$s^*$-concave $t_1$  density function $f$ (black) with its bi-$s^*$-concave  upper bounds 
    $F_U^{\prime}$ (red) and $F_L^{\prime}$ (blue) as given by (\ref{FprimeLowerSNeg}) and (\ref{FprimeUpperSNeg}).}
     \label{fig:fig2}
 \end{figure}
 
Upper and lower bounds for the derivative $f'$ of $f$ are given in (iv) of Theorem~\ref{thm:1s}:
These bounds are illustrated for the bi$-s^*-$concave distribution $t_1$ in Figure~\ref{fig:fig3}.
 
 \begin{figure}[ht]
    \centering
    \includegraphics[width=\linewidth,height=5.5cm,keepaspectratio]{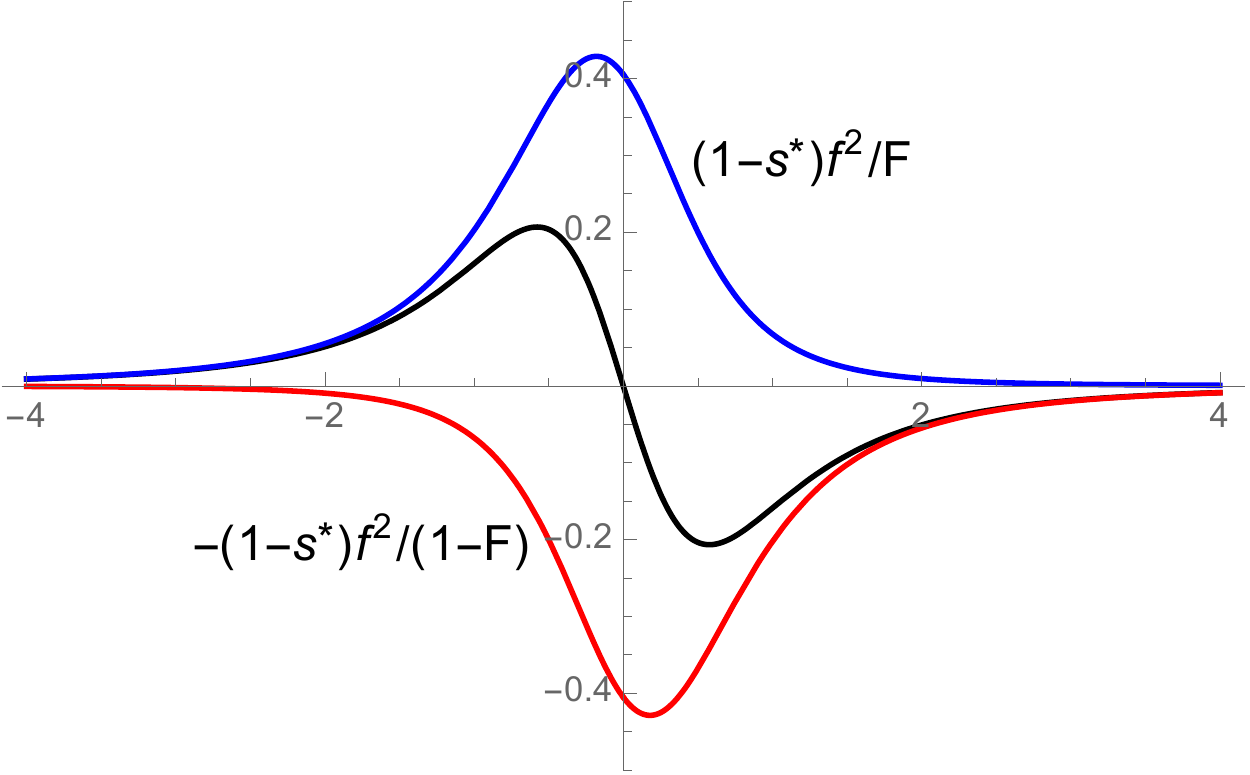} 
    \caption{The bi-$s^*$-concave $t_1$  density function derivative $f^{\prime}$ (black) with its bi-$s^*$-concave 
    lower (red) and upper (blue) bounds as given in (iv) of Theorem~\ref{thm:1s}.}
     \label{fig:fig3}
 \end{figure}
\medskip

\subsection{Bounds for $F$ bi-$s^*$-concave when $s>0$.}
\label{subsec:BoundsForSpos}

Upper and lower bounds on $F$:   Note that now $(1+y)^r \le 1+ ry$ for any $r \in (0,1]$ and $y \ge -1$ 
by concavity of $(1+y)^r$.
Taking $y = -F(x)$ and $r = s^*>0$  (since $s>0$) yields 
\begin{eqnarray*}
(1-F(x))^{s^*} \le 1- s^* F(x) . 
\end{eqnarray*}
By rearranging,
\begin{eqnarray}
F(x) & \le & \frac{1}{-s^*} \left \{ (1-F(x))^{s*} - 1 \right \} \nonumber \\
& = & \frac{1}{s^*} \left \{ 1- (1-F(x))^{s^*} \right \} \equiv F_{U,s} (x) \equiv F_U (x)
\label{FUpperSPos}
\end{eqnarray}
where $F_U$ is a convex function if $F$ is bi$-s^*-$concave.  
Similarly, taking $y= - (1-F(x))$ and $r=s^*$ yields, by rearranging terms
\begin{eqnarray}
F(x) \ge \frac{1}{s^*} \left \{ F(x)^{s^*} - (1-s^*) \right \} \equiv F_{L,s} (x) \equiv F_L (x) 
\label{FLowerSPos}
\end{eqnarray}
where $F_L$ is a concave function if $F$ is bi$-s^*-$concave.  
Note that 
\begin{eqnarray}
F_U^{\prime} (x) =  \frac{f(x)}{(1-F(x))^{1-s^*}}  = \frac{f(x)}{(1-F(x))^{1/(1+s)}} 
\label{FprimeUpperSPos}
\end{eqnarray}
is monotone non-decreasing, while
\begin{eqnarray}
F_L^{\prime} (x)  =  \frac{f(x)}{F^{1-s^*}(x)}  = \frac{f(x)}{F^{1/(1+s)} (x)}  
\label{FprimeLowerSPos}
\end{eqnarray}
is monotone non-increasing.  
  Therefore
\begin{eqnarray*}
&& 0 \le F_U^{\prime \prime} (x) =  (1-F(x))^{s^*-2} \left \{ (1-s^*) f^2 (x) + (1-F(x)) f^{\prime} (x) \right \}, \\
&& 0\ge F_L^{\prime \prime} (x)  =   F(x)^{s^*-2} \left \{ (s^*-1) f^2 (x) + F(x) f^{\prime}(x) \right \} . 
\end{eqnarray*}
Again note that the upper and lower bounds in (iv) of Theorem~\ref{thm:1s} follow by rearranging these inequalities. 

Taking $F$ to be the distribution function of  $g(\cdot, r)$ with $r=1$ as in 
Example~\ref{exmpl:ex4} and plotting the bounds  for $F$, $F^{\prime} = f$ 
and $F^{\prime \prime} = f^{\prime}$ yields the following three figures.  

\begin{figure}[ht]
    \centering
    \includegraphics[width=\linewidth,height=5.5cm,keepaspectratio]{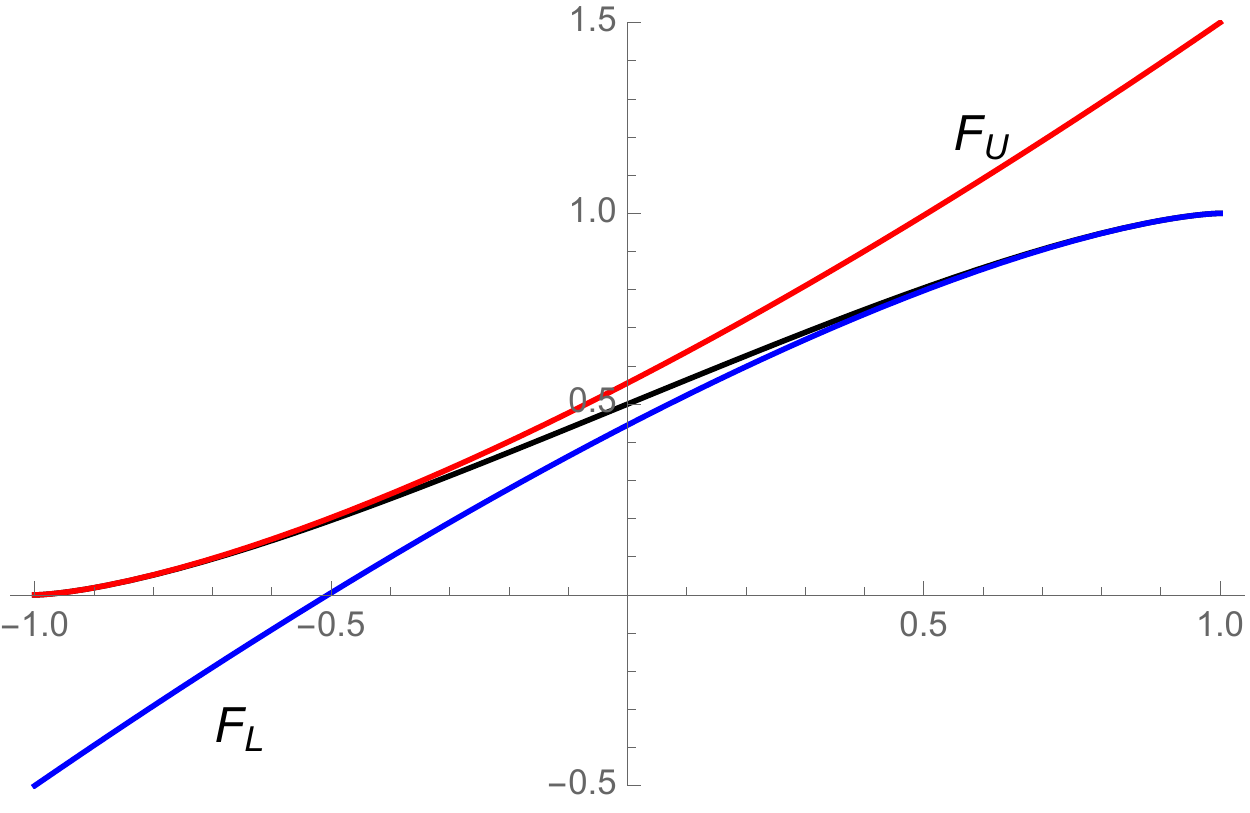}
    \caption{The bi-$s^*$-concave distribution function $F$ (black) corresponding to $g(\cdot; 1)$ of Example~\ref{exmpl:ex4}
    with its convex upper bound $F_U$ (red) and concave lower bound $F_L$ (blue)
    (where $F_U$ and $F_L$ are given in (\ref{FUpperSPos}) and (\ref{FLowerSPos})).}
     \label{fig:fig5}
 \end{figure}
Upper and lower bounds for the density $f = F^{\prime} $ of $F$ follow from (iii)  of Theorem~\ref{thm:1s}.
These bounds are illustrated for the bi-$s^*$-concave distribution $F$ 
corresponding to $g(\cdot; 1)$ of Example~\ref{exmpl:ex4}
in Figure~\ref{fig:fig6}.
 
\begin{figure}[ht]
    \centering 
   \includegraphics[width=\linewidth,height=5.5cm,keepaspectratio]{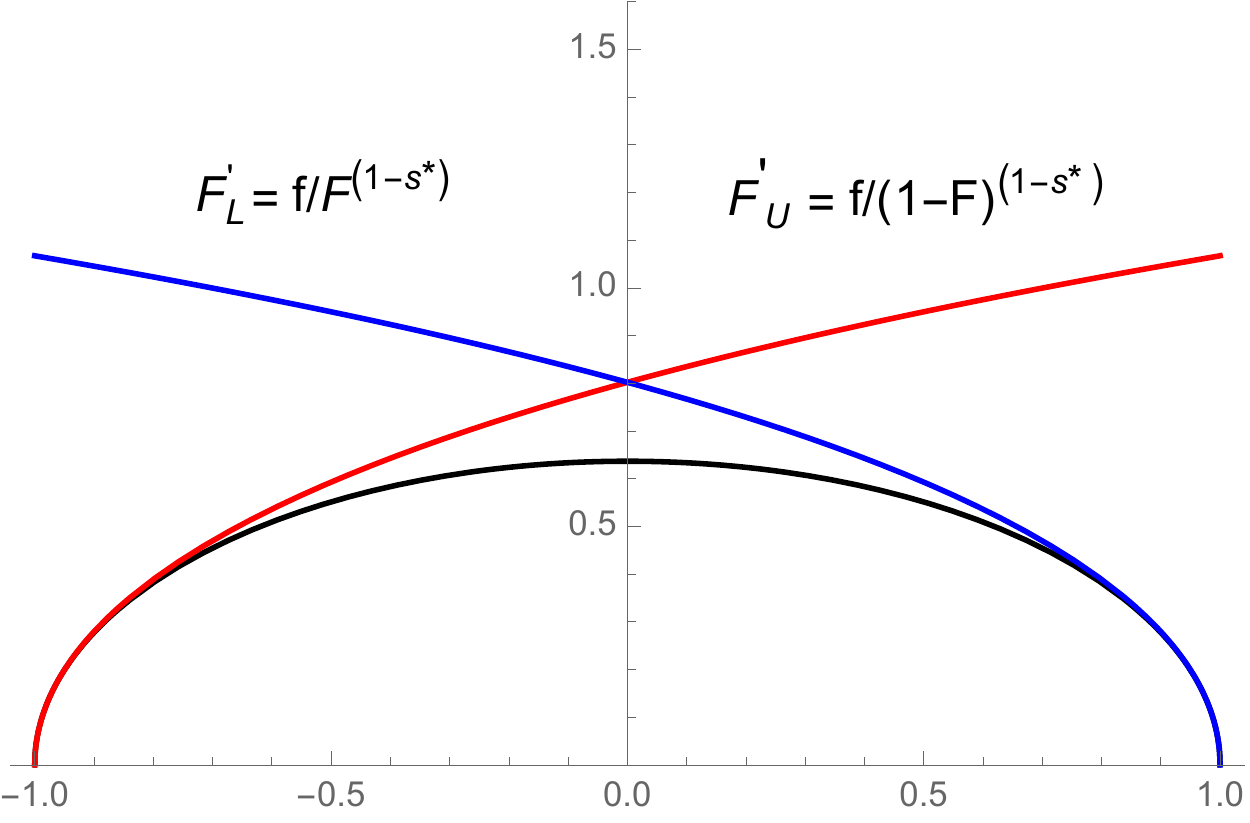} 
    \caption{The bi-$s^*$-concave density function $g(\cdot; 1)$ of Example~\ref{exmpl:ex4} (black) with its 
    bi-$s^*$-concave  upper bounds $F_L^{\prime}$ and $F_U^{\prime}$ given in (\ref{FprimeLowerSPos}) and (\ref{FprimeUpperSPos}).}
     \label{fig:fig6}
 \end{figure}

Upper and lower bounds for the derivative $f'$ of $f$ are given in (iv) of Theorem~\ref{thm:1s} 
These bounds are illustrated for the bi-$s^*$-concave distribution function $F$ with density 
$g(\cdot ; 1)$ as in Example~\ref{exmpl:ex4} in Figure~\ref{fig:fig7}.
 
 \begin{figure}[ht]
    \centering
   \includegraphics[width=\linewidth,height=5.5cm,keepaspectratio]{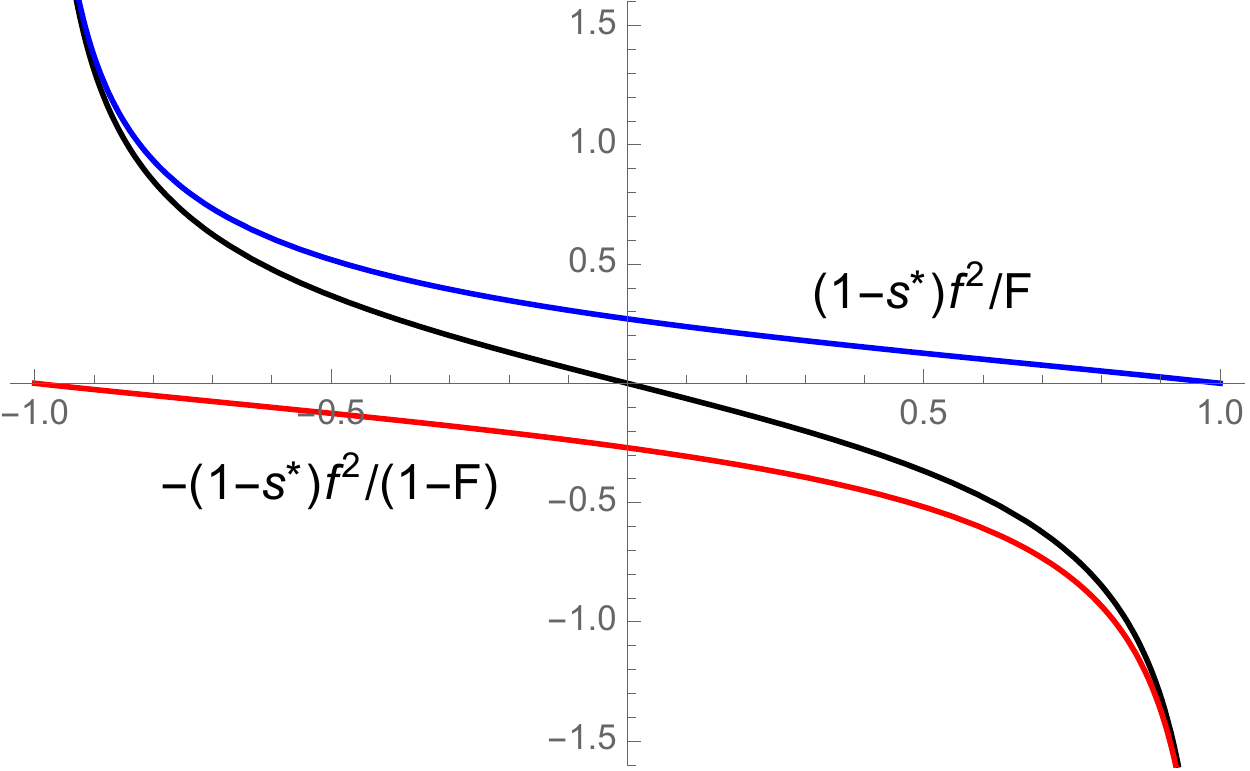} 
    \caption{$F^{\prime \prime} = f^{\prime}$ (black) for the bi-$s^*$-concave function $F$ corresponding to the 
    density $g(\cdot; 1)$ as in Example~\ref{exmpl:ex4}
    with its bi-$s^*$-concave 
     upper (blue) and lower (red) bounds  as given in (iv) of Theorem~\ref{thm:1s}.}
     \label{fig:fig7}
 \end{figure}
 \FloatBarrier
 
 \section{A consequence for Fisher information}
\label{sec:FisherInformation}
In this section we suppose that $F$ is a bi-$s^*$-concave distribution function with 
absolutely continuous density $f$ with respect to Lebesgue measure.   
Then from (\ref{4eq3}) of Theorem~\ref{thm:1s} it follows that 

\begin{eqnarray*}
\frac{| f' (x) |}{f(x)} \le \frac{1}{1+s} \frac{f(x)}{F(x) \wedge (1-F(x))}  \ \ \mbox{for all} \ \ x \in J(F),
\end{eqnarray*}
and hence that 
\begin{eqnarray}
\lefteqn{I_f \equiv \int_{\RR}  \left ( \frac{|f' (x)|}{f(x)} \right )^2 f(x) dx } \nonumber \\
&&  \le  \frac{1}{(1+s)^2}  \int_{\RR} \frac{f^2 (x)}{(F(x) \wedge (1-F(x)))^2} d F(x) \nonumber \\
&& \le  \frac{1}{(1+s)^2}  \left \{ \int_{\RR} \frac{f^2 (x)}{F^2(x)} dF(x) + \int_{\RR} \frac{f^2 (x)}{(1-F(x))^2} d F(x) \right \} \nonumber\\
&&  \le  \frac{2}{(1+s)^2} \max \left \{ \int_{\RR} \left ( \frac{f}{F} \right )^2 dF, \int_{\RR} \left ( \frac{f}{1-F} \right )^2 dF \right \} .
\label{FisherInformationUpperbounded}
\end{eqnarray}
But with $h = f'/f$, we find that 
$$
\int_{-\infty}^x  h dF = \int_{-\infty}^x (f'(y) / f(y) ) f(y) dy = f(x)  \ \ \mbox{and} \ \ \frac{f(x)}{F(x)} = \frac{\int_{-\infty}^x h dF}{F(x)} ,
$$
while 
$$
\int_x^{\infty} h dF = \int_x^{\infty} (f'/f) f dy = - f(x),  \ \ \mbox{and} \ \ \frac{-f(x)}{1-F(x)} = \frac{\int_x^\infty h dF }{1-F(x)} .
$$
Thus by the $L_2$ version of Hardy's inequality 
\begin{eqnarray}
\int_{\RR} \left ( \frac{f(x)}{F(x)} \right )^2 d F(x) 
& \le  &4  \int_{\RR} \left ( \frac{|f' (x)|}{f(x)} \right )^2 f(x) dx  = 4 I_f, \ \ \mbox{and}  \nonumber \\
              \int_{\RR} \left ( \frac{f(x)}{1-F(x)} \right )^2 d F(x) 
& \le &  4  \int_{\RR} \left ( \frac{|f' (x)|}{f(x)} \right )^2 f(x) dx  = 4 I_f .
\label{FisherInformationLowerBounded}
\end{eqnarray} 
Combining the inequalities in  (\ref{FisherInformationUpperbounded}) and (\ref{FisherInformationLowerBounded}) yields
\begin{eqnarray}
I_f & \le & \frac{2}{(1+s)^2} \max \left \{ \int_{\RR} \left ( \frac{f(x)}{F(x)} \right )^2 d F(x) , 
\ \int_{\RR} \left ( \frac{f(x)}{1-F(x)} \right )^2 d F(x) \right \}  \nonumber \\
& \le & \frac{8}{(1+s)^2} I_f .
\label{FisherInfoUpperAndLowerBounded}
\end{eqnarray}
But we note that the densities $f_r$ in Example 4 have
\begin{eqnarray*}
I_{f_r} = \frac{r}{2} \cdot \frac{\Gamma \left (\frac{r}{2}-1\right) 
\Gamma \left ( \frac{r+3}{2} \right )}{\Gamma \left ( \frac{r}{2} +1 \right )^2} \nearrow \infty
\end{eqnarray*}
as $r \searrow 2$, and $I_{f_r} = \infty$ for $0 < r \le 2$.  
In this latter case all the integrals in (\ref{FisherInfoUpperAndLowerBounded}) are infinite.

\section{Questions and further problems}
\label{sec:QuestAndProb}

\begin{question}
{\rm Application of bi-$s^*$-concavity to construction of confidence bands for $F$?}
\cite{DuembgenKW:2017}  use their bi-log-concave bounds to construct new confidence bands for bi-log-concave 
distribution functions $F$.  Alternative confidence bands based on the bi-$s^*$-concavity 
assumption may be of interest.
\end{question}
\begin{question}
{\rm What can be said when $s \le -1$?}  The only result we know in the direction of 
preserving $s-$concavity in the spirit of Borell, Brascamp and Lieb, and Rinott is due to 
\cite{MR572660}, but we do not have an interpretation of their result.  
We also do not know if there is an approximation of the general (standardized) quantile process $\QQ_n$ in terms of the uniform
quantile process $\VV_n$ in this case.
\end{question}
\begin{question}
{\rm Bi-log-concavity or bi-$s^*$-concavity in higher dimensions?}
Although log-concave (and $s-$concave) densities and measures on $\RR^d$ (and a variety of non-Euclidean spaces) exist,
we do not know of any analogue of bi$-s^*$concavity or bi-log-concavity in higher dimensions.
\end{question}
\begin{question}
{\rm Tranportation distances for empirical measures when $d\ge2$?}
The Cs\"org\H{o} - R\'ev\'esz condition has proved very useful for studying empirical transportation 
distances for empirical measures in one dimension, largely because of the connection with quantile 
processes.  We do not know of comparable theory for  transportation distances for empirical measures
in higher dimensional settings.
\end{question}

\section{Appendix:  proof of Theorem~\ref{thm:1s} when $s \in (0,\infty)$}
\label{sec:appendix}

\begin{proof} 
 Our proof of Theorem~\ref{thm:1s} for the case $s \in (0,\infty)$ closely parallels the proof 
for the case $s\in(-1,0]$.  The main difference is the proof of (iii) implies (iv).  
When $0 < s < \infty$, $s^* = s/(1+s) \in (0,1)$, and hence $1-2s^* < 0$ for $s>1$.  
This requires a slightly different argument in this range and results in different constants 
in the Lipschitz bounds. 

 Let us denote $\inf J(F)$ and $\sup J(F)$ by $a$ and $b$ respectively. Notice that $J(F)=(a,b)$ if $F$ is continuous.  

Proof of (i) implies (ii): 
Since $F$ is bi-$s^*$-concave with $\st>0$, $\psi=F^{\st}$ is concave on $(a,\infty)$. Consequently 
$\psi$, and hence $F$ also, is continuous on $(a,\infty)$ by  Lemma $6$ of \cite{DuembgenKW:2017}. 
  Similarly, the $s^*-$concavity of $1-F$ 
  implies continuity of $1-F$ on $(-\infty,b).$
 Now if $a=b$, $F$ would be degenerate. Hence, $a<b$ and $F$ is continuous on $\RR$. 
  Therefore we can also conclude that $J(F)=(a,b)$.  
  
Let $x\in(a,b)$. Concavity of $\psi$ implies that 
\begin{eqnarray*}
  F'(x\pm)=\lim_{t\to 0,\pm t>0}\dfrac{\psi^{1/s^*}(x+t)-\psi^{1/s^*}(t)}{t}=\dfrac{1}{s^*}\psi(x)^{1/s^*-1}\psi'(x\pm)
\end{eqnarray*}
  exist and satisfy
\begin{eqnarray*}
  F'(x+) \le F'(x-) .
\end{eqnarray*} 
Similarly, concavity of $(1-F)^{s^*}$ yields 
\begin{eqnarray*}
  (1-F)'(x-)\geq (1-F)'(x+)
\end{eqnarray*} 
  which implies that
\begin{eqnarray*}
  -F'(x-)\geq -F'(x+).
\end{eqnarray*}
Therefore $F'(x-)=F'(x+)$ which proves the differentiability of $F$. 
It also shows that $\psi'(x+)=\psi'(x-)=\psi'(x)$ on $(a,b)$.
 
By Lemma 6 of \cite{DuembgenKW:2017} for each $x\in(a,b)$ and $c\in[\psi'(x+),\psi'(x-)]$ one has
 \begin{eqnarray*}
    \psi(x+t) -\psi(x)\leq ct\ \text{ for }t\in(a-x,\infty)
 \end{eqnarray*}
 since $\psi$ is concave on $(a,\infty).$
    Therefore for such $x$ and $t,$
 \begin{eqnarray*}
     \psi(x+t)-\psi(x)\leq t \psi'(x).
 \end{eqnarray*} 
    Hence, 
 \begin{eqnarray*}
    F^{s^*}(x+t)-F^{s^*}(x)\leq ts^{*}f(x)F(x)^{s^*-1} 
 \end{eqnarray*}
    or, 
 \begin{eqnarray*}
    \dfrac{F^{s^*}(x+t)}{F^{s^*}(x)}\leq 1+s^*\dfrac{f(x)}{F(x)}t .
 \end{eqnarray*}
    Hence,
\begin{eqnarray*}
   \dfrac{F(x+t)}{F(x)}\leq\bigg ( 1+s^*\dfrac{f(x)}{F(x)}t\bigg )^{1/s^*}.
\end{eqnarray*}
     Analogously it follows that for $t\in(-\infty,b-x)$,
\begin{eqnarray*}
(1-F(x+t))^{s^*}-(1-F(x))^{s^*} \leq -ts^*f(x)(1-F(x))^{s^*-1}
\end{eqnarray*} 
   which yields
\begin{eqnarray*}   
  \bigg(\dfrac{1-F(x+t)}{1-F(x)}\bigg )^{s^*}\leq 1-ts^*\dfrac{f(x)}{1-F(x)}
\end{eqnarray*}
    or
 \begin{eqnarray*}
 F(x+t)\geq 1-(1-F(x))\cdot \bigg (1-ts^*\dfrac{f(x)}{1-F(x)}\bigg)^{1/s^*}.
 \end{eqnarray*}
    Hence \eqref{th12n} is proved. Notice that for $\st< 0$ the inequalities in \eqref{th12} 
    hold for all $t$ because if $\st<0$, unlike the present case, 
    $F^{\st}$ and $(1-F)^{\st}$ are convex on the entire real line.
  
Proof of (ii) implies (iii):  
Since
    (ii) holds, $F$ is continuous and differentiable on $J(F)$ with derivative $f=F'$ and satisfies \eqref{th12n}. 
    Now let $x,y\in J(F)$ with $x<y.$  Let 
\begin{eqnarray}
    h=f/F^{1-s^*}.
    \label{defhsp}
\end{eqnarray} 
 Then applying \eqref{th12n} we obtain that
\begin{eqnarray*}
    \dfrac{F^{s^*}(x)}{F^{s^*}(y)}\leq 1+s^*\dfrac{f(y)}{F(y)}(x-y).
\end{eqnarray*}
    Hence, 
\begin{eqnarray*}
    F^{s^*}(x)
    & \leq & F^{s^*}(y)+s^*\dfrac{f(y)}{F(y)^{1-s^*}}(x-y)\\
    & = & \ F^{s^*}(y)+s^*h(y)(x-y)\\
    &\leq & \ F^{s^*}(x)+s^*h(x)(y-x)+s^*h(y)(x-y).
\end{eqnarray*}
    Therefore 
\begin{eqnarray*}
    s^{*}(x-y)(h(y)-h(x))\geq 0
\end{eqnarray*} 
    where $s^{*}(x-y)<0$,  implying that $h(y)\leq h(x)$. 
    Therefore $h$ is non-increasing. Now let 
\begin{eqnarray}
    \tilde{h}=f/(1-F)^{1-s^{*}}.
     \label{defthsp}
\end{eqnarray} 
From \eqref{th12n} we also obtain that 
\begin{eqnarray*}
    (1-F(x))^{s^{*}}-(1-F(y))^{s^*}\leq -ts^{*}\dfrac{f(y)}{(1-F(y))^{1-s^*}}=-ts^{*}\tilde{h}(y)
\end{eqnarray*}
    or
\begin{eqnarray*}
  (1-F(x))^{s^{*}}
  & \leq & (1-F(y))^{s^*}  -(x-y)s^{*}\tilde{h}(y)\\
  & = & (1-F(x))^{s^*}-(y-x)s^{*}\tilde{h}(x)  -(x-y)s^{*}\tilde{h}(y)\\
  & =  & (1-F(x))^{s^*}-s^{*}(y-x)(\tilde{h}(y)-\tilde{h}(x)).\\
\end{eqnarray*}
    Since $s^{*}(y-x)>0,$ the last inequality leads to
\begin{eqnarray*}
    0\leq \tilde{h}(y)-\tilde{h}(x), 
\end{eqnarray*}  
    implying that $\tilde{h}$ is non-decreasing. 

Proof of (iii) implies (iv):
   If the conditions of (iii) hold, then it immediately follows that $f>0$ on $J(F)$. 
    If not, suppose that $f(x_0) = 0$ for some $x_0\in J(F)$ where $J(F)=(a,b)$ since $F$ is continuous.
    Then since $f(x)/F(x)^{1-s^*}$ is non-increasing, 
    $f(x)=0$ for  $x\in[x_0,b)$. Similarly since $f(x)/(1-F(x))^{1-s^*}$ is non-decreasing 
    we obtain $f(x)=0$ for $x\in(a, x_0]$.
    Therefore, $F'=0$  or $F$ is constant on $(a,b)$ or $J(F)$. Then $F$ violates the continuity 
    condition of (iii). Hence $f>0$ on $J(F)$.  
   
Suppose $h$ and $\tilde{h}$ are as defined in \eqref{defhsp} and \eqref{defthsp}. 
    Then the monotonicities of $h$ and $\tilde{h}$ imply that for any $x,x_0\in J(F),$
\begin{eqnarray*}
  f(x)= \left \{ \begin{array}{l l} 
      F^{1-s^*}(x)h(x)\leq h(x_0) &  \ \mbox{if} \ \  x\geq x_0,\\
      (1-F(x))^{1-s^*}\tilde{h}(x)\leq \tilde{h}(x_0) & \ \mbox{if} \ \  x\leq x_0.
      \end{array} \right .
\end{eqnarray*}
 Next, let $c,d \in J(F)$ with $c<d$. 
We will  bound $(f(y)-f(x))/(y-x)$ for $x,y\in J(F)$ such that $x,y\in(c,d)$ with $x \neq y$.
  This will yield local Lipschitz-continuity of $f$ on $J(F)$. 
  To this end, note that
\begin{eqnarray*}
 \dfrac{f(y)-f(x)}{y-x}
 & = & \ \dfrac{F^{1-s^*}(y)h(y)-F^{1-s^*}(x)h(x)}{y-x}\\
 & = & \  h(y)\dfrac{F^{1-s^*}(y)-F^{1-s^*}(x)}{y-x}+F^{1-s^*}(x)\dfrac{h(y)-h(x)}{y-x}\\
 & \leq & \  h(c)\dfrac{F^{1-s^*}(y)-F^{1-s^*}(x)}{y-x}\\
 & \to  & \  h(c) (1-s^*)f(x) F^{-s^*}(x)=(1-s^{*})h(c)h(x)F^{1-2s^{*}}(x)
\end{eqnarray*} 
as $y\to x$.  Here the inequality followed from the fact that 
\begin{eqnarray*}
\dfrac{h(y)-h(x)}{y-x}\leq 0
\end{eqnarray*} 
which holds since $h$ is non-increasing,
Now  since $h(x)\leq h(c)$, $s^*>0,$ $1- s^{*}>0$,  and $F(c)\leq F(x)\leq F(d)$, we find that
\begin{eqnarray}
\limsup_{y\to x}\dfrac{f(y)-f(x)}{y-x}\leq (1-s^*)h(c)^2F^{1-s^*}(d)F^{-s^*}(c)
\label{4eq1sp}
\end{eqnarray}
 for all $x\in(c,d)$.
Analogously with $\bar{F}=1-F$ we obtain that
\begin{eqnarray*}
\dfrac{f(y)-f(x)}{y-x}
& = & \dfrac{\bar{F}^{1-s^*}(y)\tilde{h}(y)-\bar{F}^{1-s^*}(x)\tilde{h}(x)}{y-x}\\
& = & \tilde{h}(y)\dfrac{\bar{F}^{1-s^*}(y)-\bar{F}^{1-s^*}(x)}{y-x}+\bar{F}^{1-s^*}(x)\dfrac{\tilde{h}(y)-\tilde{h}(x)}{y-x}\\
& \geq & \tilde{h}(y)\dfrac{\bar{F}^{1-s^*}(y)-\bar{F}^{1-s^*}(x)}{y-x}
\end{eqnarray*} 
since, by the non-decreasing property of $\tilde{h}$, for any $x,y\in J(F)$, 
\begin{eqnarray*}
   \dfrac{\tilde{h}(y)-\tilde{h}(x)}{y-x}>0.
\end{eqnarray*} 
Next observe that since $1-s^*=1/(1+s)>0$, and $\bar{F}(y)\leq \bar{F}(x)$ if $y\ge x$, 
\begin{eqnarray*}
\tilde{h}(y)\dfrac{\bar{F}^{1-s^*}(y)-\bar{F}^{1-s^*}(x)}{y-x}\geq \tilde{h}(d)\dfrac{\bar{F}^{1-s^*}(y)-\bar{F}^{1-s^*}(x)}{y-x}.
\end{eqnarray*}
Hence as $y\to x$ it follows that   
\begin{eqnarray*}
\liminf_{y\to x}\dfrac{f(y)-f(x)}{y-x}
& \geq & -\tilde{h}(d)(1-s^*)f(x)\bar{F}^{-s^*}(x)\\
& = & -\tilde{h}(d)\tilde{h}(x)(1-s^*)\bar{F}^{1-2s^*}(x).
\end{eqnarray*}
Therefore using the fact that $\tilde{h}(x)\leq \tilde{h}(d)$ and $1-s^*,s^*>0$ we conclude that 
\begin{eqnarray}
   \liminf_{y\to x}\dfrac{f(y)-f(x)}{y-x}\geq -\tilde{h}(d)^2(1-s^*)\bar{F}^{1-s^*}(c)\bar{F}^{-s^*}(d).
\label{4eq2sp}
\end{eqnarray}
Combining the above with \eqref{4eq1sp} we find that $f$ is Lipschitz-continuous on $(c,d)$ with Lipschitz-constant 
\begin{eqnarray*}
   \max\{(1-s^*)h(c)^2F^{1-s^*}(d)F^{-s^*}(c),(1-s^*)\tilde{h}(d)^2\bar{F}^{1-s^*}(c)\bar{F}^{-s^*}(d)\}.
\end{eqnarray*} 
This proves that $f$ is locally Lipschitz continuous on $J(F)$. 
Hence, $f$ is also locally absolutely continuous with $L^1$-derivative $f'$ such that
\begin{eqnarray*}
    f(y)-f(x)=\int_{x}^{y}f'(t)dt\ \text{  for all }x,y \in J(F);
\end{eqnarray*} 
hence $f'(x)$ can be chosen so that
\begin{eqnarray*}
    f'(x)\in\bigg[\liminf_{y\to x}\dfrac{f(y)-f(x)}{y-x},\limsup_{y\to x}\dfrac{f(y)-f(x)}{y-x}\bigg].
\end{eqnarray*} 
However \eqref{4eq1sp} and \eqref{4eq2sp} imply that for $c<x<d$,
\begin{eqnarray*}
  \lefteqn{
   \bigg[\liminf_{y\to x}\dfrac{f(y)-f(x)}{y-x},\limsup_{y\to x}\dfrac{f(y)-f(x)}{y-x}\bigg]}\\
&& \subset \bigg[-(1-s^*)\tilde{h}(d)^2\bar{F}^{1-s^*}(c)F^{-s^*}(d),(1-s^*)h(c)^2F^{1-s^*}(d)F^{-s^*}(c)\bigg]
\end{eqnarray*} 
Now since $f$ and $F$ are continuous and $F>0$ on $J(F),$ so are $h$ and $\tilde{h}$. 
Therefore,  letting $c,d\to x$ it follows that 
\begin{eqnarray*}
\dfrac{-(1-s^*)f(x)^2\bar{F}^{1-2s^*}(x)}{\bar{F}^{2-2s^*}(x)}\leq f'(x)\leq (1-s^*)\dfrac{f(x)^2F^{1-2s^*}(x)}{F^{2-2s^*}(x)};
\end{eqnarray*}
and this implies \eqref{4eq3}.

Proof of (iv) implies (i):
 Notice that the fact that (iii) implies (i) can be easily verified  since $f/F^{1-s^*}$ non-increasing on 
 $J(F)$ implies that   $F^{s^*}$ is concave on $J(F).$ Since $F$ is continuous, $J(F)=(a,b)$.  
 Now $F^{s^*}\in(0,1)$ on $J(F)$ and $F^{s^*}(x)=1$ for $x\geq b$. 
 Therefore $F^{s^*}$ is concave on $(a,\infty)$.
 Similarly one can show that $(1-F)^{s^*}$ is concave on $(-\infty,b)$.
 Therefore $F$ is bi-$s^*$-concave. Therefore it is enough to prove that (iv) implies (iii).

 By Lemma $7$ of \cite{DuembgenKW:2017} $h$ is non-increasing on $J(F)$ 
 if and only if for any $x\in J(F)$ the following holds:
 \begin{eqnarray*}
 \limsup_{y\to x}\dfrac{h(y)-h(x)}{y-x}\leq 0.
 \end{eqnarray*}
 Suppose $x\neq y\in J(F)$ and $r:=\min(x,y)$ and $s:=\max(x,y).$ Then it follows that
 \begin{eqnarray*}
\lefteqn{\dfrac{h(y)-h(x)}{y-x} =  \dfrac{f(y)/F^{1-s^*}(y)-f(x)/F^{1-s^*}(x)}{y-x}} \\
     && = \dfrac{1}{F^{1-s^*}(y)}\dfrac{f(y)-f(x)}{y-x}
                 -\dfrac{f(x)}{F^{1-s^*}(x)F^{1-s^*}(y)}\dfrac{F^{1-s^*}(y)-F^{1-s^*}(x)}{y-x}\\
     && = \dfrac{1}{F^{1-s^*}(y)}\dfrac{\int_{r}^{s}f'(t)dt}{s-r}
                 -\dfrac{f(x)}{F^{1-s^*}(x)F^{1-s^*}(y)}\dfrac{F^{1-s^*}(y)-F^{1-s^*}(x)}{y-x}\\
     && \leq  \dfrac{(1-s^*)}{F^{1-s^*}(y)(s-r)}\int_{r}^{s}\dfrac{f(t)^2}{F(t)}dt
                 -\dfrac{f(x)}{F^{1-s^*}(x)F^{1-s^*}(y)}\dfrac{F^{1-s^*}(y)-F^{1-s^*}(x)}{y-x}
 \end{eqnarray*}
  by \eqref{4eq3}. 
  Since $F$ is continuous by (iv), $J(F)=(a,b)$. 
  Also since $x,y\in J(F),$ $[r,s]\subset J(F)$.  Since $f$ and $F$ are continuous on 
  $J(F)$ and $F>0$ on $J(F)$,  $f^2/F$ is continuous and integrable on $J(F)$ and hence also on $[r,s]$.    
  Letting $y\to x$ we obtain that
 \begin{eqnarray*}
     \limsup_{y\to x}\dfrac{h(y)-h(x)}{y-x}\leq \dfrac{(1-s^*)f(x)^2}{F^{2-s^*}(x)}-\dfrac{(1-s^*)f(x)^2}{F^{2-s^*}(x)}=0 .
 \end{eqnarray*}
 Analogously, by Lemma $7$ of \cite{DuembgenKW:2017}, to show $\tilde{h}$ is non-decreasing it is enough to show that
 \begin{eqnarray*}
 \liminf_{y\to x}\dfrac{\tilde{h}(y)-\tilde{h}(x)}{y-x}\geq 0.
 \end{eqnarray*}
 To verify this suppose $x\neq y\in J(F)$ and $r:=\min(x,y)$ and $s:=\max(x,y)$. 
 As before we calculate
 \begin{eqnarray*}
    \dfrac{\tilde{h}(y)-\tilde{h}(x)}{y-x}
 & = &\  \dfrac{f(y)/\bar{F}^{1-s^*}(y)-f(x)/\bar{F}^{1-s^*}(x)}{y-x}\\
 & = &\  \dfrac{1}{F^{1-s^*}(y)}\dfrac{f(y)-f(x)}{y-x}-\dfrac{f(x)}{\bar{F}^{1-s^*}(x)\bar{F}^{1-s^*}(y)}
              \dfrac{\bar{F}^{1-s^*}(y)-\bar{F}^{1-s^*}(x)}{y-x}\\
 & = &\  \dfrac{1}{\bar{F}^{1-s^*}(y)}\dfrac{\int_{r}^{s}f'(t)dt}{s-r}-\dfrac{f(x)}{\bar{F}^{1-s^*}(x)\bar{F}^{1-s^*}(y)}
              \dfrac{\bar{F}^{1-s^*}(y)-\bar{F}^{1-s^*}(x)}{y-x}\\
 & \geq & \  - \dfrac{(1-s^*)}{\bar{F}^{1-s^*}(y)(s-r)}\int_{r}^{s}\dfrac{f(t)^2}{\bar{F}(t)}dt 
            -\dfrac{f(x)}{\bar{F}^{1-s^*}(x)\bar{F}^{1-s^*}(y)}\dfrac{\bar{F}^{1-s^*}(y)-\bar{F}^{1-s^*}(x)}{y-x}
 \end{eqnarray*}
 by \eqref{4eq3}. Since $f$ and $\bar{F}$ are continuous on $J(F)$,  letting $y\to x$ it follows that
\begin{eqnarray*}
    \liminf_{y\to x}\dfrac{\tilde{h}(y)-\tilde{h}(x)}{y-x} 
       \geq  -\dfrac{(1-s^*)f(x)^2}{\bar{F}^{2-s^*}(x)}+\dfrac{(1-s^*)f(x)^2}{\bar{F}^{2-s^*}(x)}=0 .
\end{eqnarray*}
\end{proof}
\medskip

\par\noindent
\section*{Acknowledgement}  
We owe thanks to Lutz D\"umbgen for pointing out a simpler proof of Proposition~\ref{prop:BagnoliBergstrom}
(not given here) and for noting several typos.

\bibliographystyle{natbib}
\bibliography{BiS-concave}

\begin{thebibliography}{}

\bibitem[An(1998)An]{MR1637480}
An, M.~Y. (1998).
\newblock Logconcavity versus logconvexity: a complete characterization.
\newblock {\em J. Econom. Theory\/}, {\bf 80}(2), 350--369.

\bibitem[Bagnoli and Bergstrom(2005)Bagnoli and Bergstrom]{MR2213177}
Bagnoli, M. and Bergstrom, T. (2005).
\newblock Log-concave probability and its applications.
\newblock {\em Econom. Theory\/}, {\bf 26}(2), 445--469.

\bibitem[Barlow and Proschan(1975)Barlow and Proschan]{MR0438625}
Barlow, R.~E. and Proschan, F. (1975).
\newblock {\em Statistical {T}heory of {R}eliability and {L}ife {T}esting\/}.
\newblock Holt, Rinehart and Winston, Inc., New York-Montreal, Que.-London.
\newblock Probability models, International Series in Decision Processes,
  Series in Quantitative Methods for Decision Making.

\bibitem[Bobkov and Ledoux(2017)Bobkov and Ledoux]{Bobkov-Ledoux:2017}
Bobkov, S. and Ledoux, M. (2017).
\newblock {\em One-dimensional empirical measures, order statistics, and
  Kantorovich transport distances\/}.
\newblock Memoirs of the American Mathematical Society. American Mathematical
  Society.

\bibitem[Borell(1975)Borell]{MR0404559}
Borell, C. (1975).
\newblock Convex set functions in {$d$}-space.
\newblock {\em Period. Math. Hungar.}, {\bf 6}(2), 111--136.

\bibitem[Boyd and Vandenberghe(2004)Boyd and Vandenberghe]{MR2061575}
Boyd, S. and Vandenberghe, L. (2004).
\newblock {\em Convex {O}ptimization\/}.
\newblock Cambridge University Press, Cambridge.

\bibitem[Brascamp and Lieb(1976)Brascamp and Lieb]{MR0450480}
Brascamp, H.~J. and Lieb, E.~H. (1976).
\newblock On extensions of the {B}runn-{M}inkowski and {P}r\'ekopa-{L}eindler
  theorems, including inequalities for log concave functions, and with an
  application to the diffusion equation.
\newblock {\em J. Functional Analysis\/}, {\bf 22}(4), 366--389.

\bibitem[Cs\"org\H{o} and R\'ev\'esz(1978)Cs\"org\H{o} and
  R\'ev\'esz]{MR0501290}
Cs\"org\H{o}, M. and R\'ev\'esz, P. (1978).
\newblock Strong approximations of the quantile process.
\newblock {\em Ann. Statist.}, {\bf 6}(4), 882--894.

\bibitem[Dancs and Uhrin(1980)Dancs and Uhrin]{MR572660}
Dancs, S. and Uhrin, B. (1980).
\newblock On a class of integral inequalities and their measure-theoretic
  consequences.
\newblock {\em J. Math. Anal. Appl.}, {\bf 74}(2), 388--400.

\bibitem[del Barrio {\em et~al.}(2005)del Barrio, Gin\'e, and Utzet]{MR2121458}
del Barrio, E., Gin\'e, E., and Utzet, F. (2005).
\newblock Asymptotics for {$L_2$} functionals of the empirical quantile
  process, with applications to tests of fit based on weighted {W}asserstein
  distances.
\newblock {\em Bernoulli\/}, {\bf 11}(1), 131--189.

\bibitem[Dharmadhikari and Joag-Dev(1988)Dharmadhikari and Joag-Dev]{MR954608}
Dharmadhikari, S. and Joag-Dev, K. (1988).
\newblock {\em Unimodality, {C}onvexity, and {A}pplications\/}.
\newblock Probability and Mathematical Statistics. Academic Press, Inc.,
  Boston, MA.

\bibitem[Dierker(1991)Dierker]{Dierker:1991}
Dierker, E. (1991).
\newblock Competition for consumers.
\newblock In {\em Equilibrium theory and applications\/}, pages 393 -- 402.
  Cambridge University Press.

\bibitem[D\"umbgen {\em et~al.}(2017)D\"umbgen, Kolesnyk, and
  Wilke]{DuembgenKW:2017}
D\"umbgen, L., Kolesnyk, P., and Wilke, R. (2017).
\newblock Bi-log-concave distribution functions.
\newblock {\em J. Statist. Planning and Inference\/}, {\bf 184}, 1--17.

\bibitem[Erd\H{o}s and Stone(1970)Erd\H{o}s and Stone]{MR0260958}
Erd\H{o}s, P. and Stone, A.~H. (1970).
\newblock On the sum of two {B}orel sets.
\newblock {\em Proc. Amer. Math. Soc.}, {\bf 25}, 304--306.

\bibitem[Gardner(2002)Gardner]{MR1898210}
Gardner, R.~J. (2002).
\newblock The {B}runn-{M}inkowski inequality.
\newblock {\em Bull. Amer. Math. Soc. (N.S.)\/}, {\bf 39}(3), 355--405.

\bibitem[Prekopa(1973)Prekopa]{MR0404557}
Prekopa, A. (1973).
\newblock On logarithmic concave measures and functions.
\newblock {\em Acta Sci. Math. (Szeged)\/}, {\bf 34}, 335--343.

\bibitem[Rinott(1976)Rinott]{MR0428540}
Rinott, Y. (1976).
\newblock On convexity of measures.
\newblock {\em Ann. Probability\/}, {\bf 4}(6), 1020--1026.

\bibitem[Shorack and Wellner(1986)Shorack and Wellner]{MR838963}
Shorack, G.~R. and Wellner, J.~A. (1986).
\newblock {\em Empirical {P}rocesses with {A}pplications to {S}tatistics\/}.
\newblock Wiley Series in Probability and Mathematical Statistics: Probability
  and Mathematical Statistics. John Wiley \& Sons Inc., New York.

\bibitem[Shorack and Wellner(2009)Shorack and Wellner]{MR3396731}
Shorack, G.~R. and Wellner, J.~A. (2009).
\newblock {\em Empirical {P}rocesses with {A}pplications to {S}tatistics\/},
  volume~59 of {\em Classics in Applied Mathematics\/}.
\newblock Society for Industrial and Applied Mathematics (SIAM), Philadelphia,
  PA.
\newblock Reprint of the 1986 original [ MR0838963].

\end{thebibliography}

\end{document}